\documentclass{article}
\usepackage[
  journal=MSL,
  lang=british,
]{ems-journal}

\makeatletter
\@ifl@t@r\fmtversion{2024-11-01}{%
  \def\H@refstepcounter#1{%
    \stepcounter{#1}%
    \edef\@currentcounter{#1}%
    \protected@edef\@currentlabel{%
      \csname p@#1\expandafter\endcsname\csname the#1\endcsname}%
  }%
}{}
\makeatother

\usepackage{mathtools}
\usepackage{bm}
\usepackage{enumitem}
\usepackage{microtype}

\theoremstyle{plain}
\newtheorem{theorem}{Theorem}[section]
\newtheorem{proposition}[theorem]{Proposition}
\newtheorem{lemma}[theorem]{Lemma}
\newtheorem{corollary}[theorem]{Corollary}
\theoremstyle{definition}
\newtheorem{definition}[theorem]{Definition}
\newtheorem{example}[theorem]{Example}
\theoremstyle{remark}
\newtheorem{remark}[theorem]{Remark}
\numberwithin{equation}{section}

\newcommand{\PP}{\mathbb P}
\newcommand{\EE}{\mathbb E}
\newcommand{\R}{\mathbb R}
\newcommand{\1}{\mathbf 1}
\newcommand{\diag}{\operatorname{diag}}
\newcommand{\pref}{\preceq}
\newcommand{\cI}{\mathcal I}

\begin{document}

\title{A proof of Ross's conjecture for two-site moving-target search}

\emsauthor{1}{
  \givenname{Yunpeng}
  \surname{Li}
  \mrid{}
  \orcid{}}{Y.~Li}

\Emsaffil{1}{
  \department{School of Data Science}
  \organisation{The Chinese University of Hong Kong, Shenzhen}
  \rorid{}
  \address{}
  \zip{}
  \city{}
  \country{China}
  \affemail{liyunpeng@cuhk.edu.cn}
}

\classification[90C40, 93E20]{90B40}

\begin{abstract}
A target moves between two sites according to a discrete-time Markov
chain with a $2\times2$ transition matrix $M$.  At each epoch one site is searched
at positive cost, and a search may overlook a target that is present.
Ross conjectured that an optimal policy is threshold in the posterior
probability that the target is at site~1.  MacPhee and Jordan proved the
conjecture throughout the nonpositive-determinant ($\det M\le0$) regime
and for part of the positive-determinant ($\det M>0$) regime, leaving
the remaining cases open.  We prove threshold optimality throughout the
positive-determinant regime, completing Ross's conjecture for
all parameter values.
\end{abstract}

\maketitle
\thispagestyle{empty}

\section{Introduction}
\label{sec:introduction}

Search for a target whose location evolves while the search is in
progress is a classical problem in stochastic control and search theory;
see Brown~\cite{brown1980} and the survey of Benkoski, Monticino and
Weisinger~\cite{benkoski1991}.  The two-region Markovian model underlying
the present problem was introduced by Pollock~\cite{pollock1970}, and
Schweitzer~\cite{schweitzer1971} gave a recursive procedure for computing
its threshold.  Ross later formulated the general threshold conjecture
studied here~\cite[Chapter~III, Section~5]{ross1983}.  A target occupies
one of two sites and moves after each unsuccessful search according to a
two-state Markov chain.  Write $a$ and
$b$ for the one-step probabilities of remaining at sites~1 and~2,
respectively; the determinant of the transition matrix is then
$\Delta=a+b-1$.  Searching site $i$ costs $C_i>0$; if the target is there, the
search fails to detect it with probability $\alpha_i\in[0,1]$.
Ross's conjecture is usually stated for $\alpha_i<1$, so that each site
is in principle detectable.  Let $X_t\in\{1,2\}$ denote the
target's location at the beginning of search epoch $t$, and let
$\mathcal H_t$ be the $\sigma$-field generated by all search decisions
and observations made before the search at that epoch is chosen.  The
sufficient statistic is the scalar posterior belief
\begin{equation}
 p_t:=\PP(X_t=1\mid\mathcal H_t).
 \label{eq:belief}
\end{equation}
At a generic search epoch we write $p$ for the current value of $p_t$.
Ross conjectured that the optimal decision is monotone in this belief:
there should be a threshold such that the low-belief side is searched at
site~2 and the high-belief side at site~1.

Although the belief state is one-dimensional, the conjecture is
substantially more delicate than it appears.  If a search fails, the
posterior first changes because of the action-dependent overlook
probability and then changes again because the target moves.  The two
Bayesian update maps are linear fractional rather than affine.  The
infinite-horizon value is an infimum of affine open-loop costs and can
have infinitely many linear pieces.  Thus a direct attempt to prove that
the difference of the two Bellman branches is monotone leads back to the
unknown value function.

White~\cite{white1992} gave an early partial proof of Ross's
conjecture together with structural machinery for the problem.  MacPhee
and Jordan~\cite{macphee1995} later gave the most extensive earlier
analysis of the discrete-time two-site model with overlooking.  Under
their standing transition assumption $0<a,b<1$, their Theorem~5
explicitly includes every transition law with $a+b\le1$, as well as
several further regions when $a+b>1$.  With the convention for $a$ and
$b$ above, their quantity $a-(1-b)$ is exactly our determinant
$\Delta=a+b-1$.  Their work therefore settled the entire strict
negative-determinant regime.
Jordan's subsequent thesis~\cite[Sections~4.4--4.5]{jordan1997}
develops the same case analysis and
again identifies unresolved subcases of the positive-determinant
class~4.2.  The later account of Flesch, Karag\"ozo\u{g}lu and
Perea~\cite{flesch2009}
explicitly states that the discrete-time conjecture was still unproved
in full.  Yu and Ye~\cite{yuye2015} subsequently studied the validity and
optimality of search strategies and the infimum expected number of
searches, completely resolving the last question and obtaining partial
results on the first two.
Weber~\cite{weber1986} proved an analogue for a continuous-time moving target,
and Assaf and Sharlin-Bilitzky~\cite{assaf1994} studied a related
continuous-time effort-allocation model.

The problem is also a small partially observed Markov decision process
(POMDP).  Finite-horizon POMDP values are piecewise linear by the
classical theory of Smallwood and Sondik~\cite{smallwood1973}, and a
large literature gives sufficient conditions for monotone or threshold
policies; see Lovejoy~\cite{lovejoy1987} and the structural treatment in
Krishnamurthy~\cite[Chapters~10--12]{krishnamurthy2016}.  Such results are powerful when
transition, observation and cost orders align in the required
supermodular or monotone-likelihood-ratio fashion.  In Ross's model the
observation kernel is itself action dependent, and the unresolved
parameter range is precisely where the standard one-step comparisons do
not give the required global single-crossing property.
Despite its elementary formulation, the problem captures important
applications.  For example, Johnston and
Krishnamurthy~\cite{johnston2006} reformulated an opportunistic
file-transfer problem over a Gilbert--Elliott fading channel as a
Markovian search problem and obtained threshold policies by applying
MacPhee and Jordan's results.
These connections make a complete resolution of Ross's conjecture
useful beyond the original search formulation.

Our main result proves Ross's conjecture throughout its classical
parameter range.

\begin{theorem}[Ross's conjecture]
\label{thm:ross}
For every $a,b\in[0,1]$, every
$\alpha_1,\alpha_2\in[0,1)$, and every $C_1,C_2>0$, there exist
a threshold $p^*\in[0,1]$ and an optimal policy of the following form:
search site~2 whenever $p<p^*$, and search site~1 whenever $p>p^*$.
At $p=p^*$, either site may be selected.
\end{theorem}

Thus the conjecture holds for every two-site transition matrix.  The
part not covered by prior work consists of the unresolved subcases of
MacPhee and Jordan's positive-determinant class~4.2
\cite[Sections~4--5]{macphee1995}. 

The rest of the paper is organised as follows.  Section~\ref{sec:model}
formulates the problem and its Bellman equation.
Sections~\ref{sec:words}--\ref{sec:infinite} prove threshold optimality
throughout this regime,
including that class.  Section~\ref{sec:negative} treats $\Delta=0$ and
states the threshold theorem for $\Delta<0$; Appendix~\ref{app:negative}
provides a self-contained, reorganised proof of that theorem.
Section~\ref{sec:boundary} treats the remaining boundary parameter
values: transition endpoints $a,b\in\{0,1\}$, perfect detection
$\alpha_i=0$, and completely ineffective searches $\alpha_i=1$.  It
also determines exactly when the expected total cost is finite.
Finally, Section~\ref{sec:discussion} compares the arguments for the
positive- and negative-determinant regimes and places the new method in
the context of structural POMDP theory and related word-based methods
in stochastic control.

\section{Model and Bellman formulation}
\label{sec:model}

At the beginning of each search epoch the target is at site $1$ or
site $2$.  If it is at site $i$ and site $i$ is searched, it is missed
with probability $\alpha_i\in[0,1]$\footnote{The value $\alpha_i=1$
means that searching site $i$ can never detect the target.
Sections~\ref{sec:words}--\ref{sec:negative} treat the proper-detection
range $\alpha_i<1$; the endpoints $\alpha_i=1$ are classified in
Section~\ref{sec:boundary}.}; a search of the wrong site also fails.
Search $i$ costs $C_i>0$.  After an unsuccessful search, the
target moves according to
\begin{equation}
 M=
 \begin{pmatrix}
 a&1-a\\
 1-b&b
 \end{pmatrix},
 \qquad 0\le a,b\le1.
 \label{eq:M}
\end{equation}
Thus $a$ and $b$ are the one-step probabilities of remaining at sites
$1$ and $2$, respectively.  We write
\begin{equation}
 \Delta:=\det M=a+b-1.
 \label{eq:Delta}
\end{equation}
The case $\Delta>0$ is the positive-determinant regime.

At a generic search epoch, write $p$ for the posterior belief defined
in \eqref{eq:belief}.  The probabilities of
an unsuccessful search at sites~1 and~2, respectively, are
\begin{equation}
 q_1(p)=\alpha_1p+1-p,
 \qquad
 q_2(p)=p+\alpha_2(1-p).
 \label{eq:q}
\end{equation}
Conditional on an unsuccessful search and the subsequent target
movement, the next beliefs after searching sites~1 and~2, respectively,
are
\begin{align}
 F_1(p)
 &=\frac{\alpha_1ap+(1-b)(1-p)}{\alpha_1p+1-p},
 \label{eq:F1}\\
 F_2(p)
 &=\frac{ap+\alpha_2(1-b)(1-p)}{p+\alpha_2(1-p)}.
 \label{eq:F2}
\end{align}
We write
\[
 \theta=(a,b,\alpha_1,\alpha_2,C_1,C_2)
\]
for the parameter vector.  We first analyse the strict-interior case
$0<a,b<1$ and $0<\alpha_1,\alpha_2<1$, in which the denominators in
\eqref{eq:F1}--\eqref{eq:F2} are positive.  Section~\ref{sec:boundary}
later treats endpoint parameters, where a denominator $q_i(p)$ may
vanish.  In that case $F_i(p)$ itself need not be defined, but, for
bounded $f$, the continuation term $q_i(p)f(F_i(p))$ extends uniquely
and continuously by setting it equal to zero.

For a fixed parameter vector $\theta$, let $V(p)\in[0,\infty]$ denote
the infimum, over all search policies including history-dependent and
randomised policies, of the expected total search cost until detection.
We assign cost $+\infty$ to a sample path on which detection never
occurs.  We call $\theta$ \emph{proper} if $V(p)<\infty$ for every
$p\in[0,1]$.

For $\alpha_1,\alpha_2<1$, the dynamic-programming equation that we will
prove for $V$ is
\begin{equation}
 V(p)=\min_{i\in\{1,2\}}
 \left\{C_i+q_i(p)V(F_i(p))\right\}.
 \label{eq:bellman}
\end{equation}
Because the total-cost criterion is undiscounted, the validity of
\eqref{eq:bellman} and the optimality of a policy that, for each current
posterior $p$, chooses a search site attaining the minimum on the
right-hand side both require proof.
Section~\ref{sec:infinite} constructs a uniformly bounded
reference policy, proves uniform convergence of the finite-horizon
values to $V$, establishes \eqref{eq:bellman}, and verifies that such a
policy is optimal.

For a bounded function $f$ on $[0,1]$, define the Bellman operator
\begin{equation}
 (\mathcal T f)(p)
 :=\min_{i\in\{1,2\}}\{C_i+q_i(p)f(F_i(p))\}.
 \label{eq:bellman-operator}
\end{equation}
Throughout the paper, a \emph{continuation} is the policy followed from
the next search epoch onward, conditional on the current search failing
and the target subsequently moving.  In value-function arguments, a
continuation is represented by a cost-to-go function $f$ of the next
belief.  Thus
$\mathcal B_i f(p):=C_i+q_i(p)f(F_i(p))$ is the expected cost of
searching site~$i$ now and, if that search is unsuccessful, following
the continuation represented by $f$.  Two current actions have the
\emph{same continuation} when the same future policy---or, at the
value-function level, the same function $f$---is used after failure
under either action.
Once \eqref{eq:bellman} has been established, define the two Bellman
branch costs
\begin{equation}
 Q_i(p):=C_i+q_i(p)V(F_i(p)),\qquad i=1,2,
 \label{eq:Q-infinite}
\end{equation}
and their action differential
\begin{equation}
 D(p):=Q_1(p)-Q_2(p).
 \label{eq:D}
\end{equation}
Site~1 is strictly optimal when $D(p)<0$ and site~2 is strictly optimal
when $D(p)>0$.

\begin{definition}[Threshold selector]
A deterministic stationary selector $\mu:[0,1]\to\{1,2\}$ is a
threshold selector if there is $p^*\in[0,1]$ such that
\[
 p<p^*\Longrightarrow\mu(p)=2,
 \qquad
 p>p^*\Longrightarrow\mu(p)=1.
\]
At $p=p^*$ either action is permitted.  More generally, if an interval
of beliefs consists entirely of ties, a threshold may be placed at any
point of that interval and the ties resolved monotonically.
\end{definition}

This formulation includes degenerate cases in which the same action is
optimal throughout the relevant belief interval.

\begin{theorem}[Positive-determinant case]
\label{thm:positive}
Suppose
\begin{equation}
 0<a,b<1,
 \qquad
 0<\alpha_1,\alpha_2<1,
 \qquad
 C_1,C_2>0,
 \qquad
 \Delta>0.
 \label{eq:positive-strict}
\end{equation}
Then the infinite-horizon problem \eqref{eq:bellman} admits an optimal
deterministic stationary threshold selector.
\end{theorem}

Theorem~\ref{thm:positive} establishes threshold optimality when
$0<a,b<1$, $0<\alpha_1,\alpha_2<1$, and $\Delta>0$.  In particular, it
resolves the unresolved subcases of MacPhee and Jordan's
positive-determinant class~4.2~\cite[Sections~4--5]{macphee1995}.

We will first analyse the finite-horizon problem, for which we introduce
the following value functions.  Set $V_0\equiv0$ and define recursively
\begin{equation}
 V_{n+1}(p)=\min_{i\in\{1,2\}}
 \left\{C_i+q_i(p)V_n(F_i(p))\right\},
 \qquad n\ge0.
 \label{eq:finite-bellman}
\end{equation}
Thus $V_n$ is the minimum expected cost accumulated until detection or
until $n$ searches have been performed, whichever occurs first, with zero
terminal cost if the target remains undetected.
Define the finite-horizon branch costs
\begin{equation}
 Q_{n+1,i}(p):=C_i+q_i(p)V_n(F_i(p)),\qquad i=1,2,
 \label{eq:Qn}
\end{equation}
and set
\begin{equation}
 D_{n+1}(p):=Q_{n+1,1}(p)-Q_{n+1,2}(p).
 \label{eq:Dn}
\end{equation}
Our finite-horizon target is the one-sided single-crossing property
\begin{equation}
 p<p',
 \qquad D_n(p)<0
 \quad\Longrightarrow\quad
 D_n(p')\le0.
 \label{eq:finite-singlecross}
\end{equation}
It is exactly what is needed for a monotone choice of minimiser: as the
belief increases, strict preference may switch from site~2 to site~1,
but it can never switch back.

\section{Joint non-detection probabilities and search words}
\label{sec:words}

This section introduces two representations used throughout the proof.
The first uses matrices to track separately the probability that the
target has not yet been detected and is at site~1, and the probability
that it has not yet been detected and is at site~2.  The second uses
finite words over the alphabet $\{1,2\}$ to record the successive sites
searched on a path along which every search so far has been unsuccessful.

\subsection{Matrix updates for joint non-detection probabilities}

After any sequence of searches, let $\nu=(\nu_1,\nu_2)$, where $\nu_j$
is the probability that the target has not yet been detected and is
currently at site~$j$, for $j=1,2$.  Before the first search these
probabilities are $p$ and $1-p$, so let
\[
 \rho(p):=(p,1-p)
\]
be the corresponding initial row vector.  Define
\begin{equation}
 A_1=\diag(\alpha_1,1)M
 =\begin{pmatrix}
 \alpha_1a&\alpha_1(1-a)\\
 1-b&b
 \end{pmatrix},
 \label{eq:A1}
\end{equation}
\begin{equation}
 A_2=\diag(1,\alpha_2)M
 =\begin{pmatrix}
 a&1-a\\
 \alpha_2(1-b)&\alpha_2b
 \end{pmatrix}.
 \label{eq:A2}
\end{equation}
Suppose the current joint-probability vector is $\nu$.  If site~$i$ is
searched, then the $j$th component of $\nu A_i$ is the probability that
the target remains undetected after that search and is at site~$j$ after the
subsequent movement.

To express search costs in the same matrix notation, let
\begin{equation}
 E_1=\begin{pmatrix}1&0\\1&0\end{pmatrix},
 \qquad
 E_2=\begin{pmatrix}0&1\\0&1\end{pmatrix},
 \qquad
 c=\binom{C_1}{C_2}.
 \label{eq:E}
\end{equation}
Let $\1=(1,1)^\top$.  Then $E_i c=C_i\1$, and hence
\[
 \nu E_ic=C_i(\nu_1+\nu_2).
\]
Here $\nu_1+\nu_2$ is the probability that the process reaches the
current search without prior detection.  Thus $\nu E_ic$ is the
expected cost contributed by searching site~$i$ at this epoch.

\subsection{Search words and notation}

A \emph{word} is simply a finite string of symbols from the alphabet
$\{1,2\}$~\cite{lothaire1997}.  In this paper the symbols are search
actions, so a word $w=w_1\cdots w_n$ means that, conditional on still
not having detected the target, search $w_k$ is used at epoch $k$.
We write $|w|=n$ for its length and $\varnothing$ for the empty word.
If $1\le k\le n$, then $w_{1:k}=w_1\cdots w_k$ is the length-$k$
prefix.  Concatenation is denoted by juxtaposition: if $u$ and $v$ are
words, then $uv$ is the word obtained by writing $v$ after $u$.  We
write $u\pref w$ when $u$ is a prefix of $w$.  A \emph{prefix chain}
is a collection of words totally ordered by $\pref$; a deepest member
is one of maximal length.

For a word $w=w_1\cdots w_n$, write
\begin{equation}
 A_w:=A_{w_1}\cdots A_{w_n},
 \qquad A_\varnothing:=I,
 \label{eq:Aw}
\end{equation}
and
\begin{equation}
 R_w
 :=\sum_{k=1}^n
 A_{w_1\cdots w_{k-1}}E_{w_k},
 \qquad R_\varnothing:=0.
 \label{eq:Rw}
\end{equation}
The matrix $R_w$ is the accumulated cost operator of the word.

\begin{lemma}[Word cost]
\label{lem:word-cost}
The expected search cost of following the sequence $w$ until either the
target is detected or all letters of $w$ have been used is
\begin{equation}
 J_w(p)=\rho(p)R_wc.
 \label{eq:Jw}
\end{equation}
In particular, $J_w$ is affine in $p$.
\end{lemma}

\begin{proof}
Before search $k$, the probabilities that the target has not yet been
detected and is at sites~1 and~2, respectively, form the row vector
\[
 \rho(p)A_{w_1\cdots w_{k-1}}.
\]
The expected cost paid at that search is therefore
\[
 \rho(p)A_{w_1\cdots w_{k-1}}E_{w_k}c.
\]
Summing over $k$ gives \eqref{eq:Jw}.
\end{proof}

The matrices satisfy the concatenation identities
\begin{equation}
 A_{uv}=A_uA_v,
 \qquad
 R_{uv}=R_u+A_uR_v.
 \label{eq:concat}
\end{equation}

\subsection{Finite-horizon open-loop reduction}

An \emph{open-loop policy} for a horizon of $n$ searches fixes a word
$w=w_1\cdots w_n$ in advance and, provided the target has not yet been
detected, searches site~$w_k$ at epoch $k$.  The next proposition shows
that allowing the policy to depend on past observations or to randomise
does not improve the finite-horizon value.

\begin{proposition}[Open-loop representation]
\label{prop:open-loop}
For every $n\ge0$,
\begin{equation}
 V_n(p)=\min_{w\in\{1,2\}^n}J_w(p).
 \label{eq:Vn-word}
\end{equation}
Consequently $V_n$ is a continuous, piecewise-affine concave function
of $p$.
\end{proposition}

\begin{proof}
The case $n=0$ is immediate.  Fix $n\ge1$ and $p\in[0,1]$.  If the
process reaches epoch $k$, then the preceding $k-1$ searches were all
unsuccessful.  The nonterminal history is therefore determined by the
sequence of sites previously searched.  With
$\{1,2\}^0:=\{\varnothing\}$, a deterministic $n$-stage policy is a
tuple $\pi=(\pi_1,\ldots,\pi_n)$ of maps
\[
 \pi_k:\{1,2\}^{k-1}\longrightarrow\{1,2\},
 \qquad 1\le k\le n,
\]
where $\pi_k(i_1,\ldots,i_{k-1})$ is the site searched at epoch $k$
after sites $i_1,\ldots,i_{k-1}$ have been searched unsuccessfully.
This policy determines a word $w^\pi=w_1^\pi\cdots w_n^\pi$ recursively
by
\[
 w_1^\pi:=\pi_1(\varnothing),
 \qquad
 w_k^\pi:=\pi_k(w_1^\pi,\ldots,w_{k-1}^\pi),
 \quad 2\le k\le n.
\]
Lemma~\ref{lem:word-cost} gives its expected $n$-horizon cost as
$J_{w^\pi}(p)$.  Conversely, every $w\in\{1,2\}^n$ defines a feasible
open-loop policy.

For a randomised policy, proceed recursively along the history in which
all $n$ searches are unsuccessful, drawing each action according to the
distribution specified by the policy.  Let $r_w$ be the probability
that the resulting word is $w$.  Then
\[
 r_w\ge0,
 \qquad
 \sum_{w\in\{1,2\}^n}r_w=1.
\]
By conditioning on the resulting word and applying
Lemma~\ref{lem:word-cost}, the expected cost accumulated until detection
or until $n$ searches have been performed, whichever occurs first, is
the convex combination $\sum_w r_wJ_w(p)$.  Consequently,
\[
 \sum_{w\in\{1,2\}^n}r_wJ_w(p)
 \ge \min_{w\in\{1,2\}^n}J_w(p).
\]
Since every word is feasible, this proves \eqref{eq:Vn-word}.  Finally,
Lemma~\ref{lem:word-cost} shows that every $J_w$ is affine in $p$; the
minimum of the finite family $\{J_w:w\in\{1,2\}^n\}$ is therefore
continuous, piecewise affine and concave.
\end{proof}

A related representation of the value as the lower envelope of the costs of
fixed search sequences appears in Jordan~\cite[Section~3.3]{jordan1997}.
What will matter here is the matrix structure of the word costs.

\section{Height-one words and a common projective separator}
\label{sec:separator}

We now prove the algebraic-combinatorial core of the paper.  The section
is self-contained, but the following roadmap may help readers who do
not usually work with words or total positivity.

\begin{enumerate}[label=\textbf{Step \arabic*.},leftmargin=*]
\item A pair of words is called \emph{height-one} when the cumulative
excess number of $1$'s in the first word over the second is always $0$
or $1$.  Plotting this discrepancy against the prefix length gives a
path on a two-level ladder.
\item Such a pair can be converted from one word to the other
using local adjacent swaps $21\to12$ (and possibly one $2\to1$
replacement).  Crucially, the prefixes immediately to the left of these
operations are nested.
\item A positive $2\times2$ matrix with positive determinant maps the
positive cone to an ordered interval of projective slopes.  The interval
shrinks when the word prefix is extended.
\item Each local word operation changes the word cost at the two
endpoint beliefs $p=0$ and $p=1$.  A suitable slope in the projective
interval of its left prefix bounds the change at $p=1$ by a positive
multiple of the change at $p=0$.  Because the operation prefixes are
nested, one slope chosen from the deepest prefix gives such a bound for
every operation.  Summing the bounds rules out an upward crossing of the
total word-cost difference.
\end{enumerate}

The terminology of finite words is standard~\cite{lothaire1997}.  The
adjacent inversion swaps used here are closely related to the local
moves underlying weak orders on
permutations~\cite[Chapter~3]{bjornerbrenti2005};
however, the chain-of-prefixes refinement proved below is specific to
the height-one condition.

\subsection{Height-one pairs and their two-level path}

For a word $w$, let $N_1(w)$ denote the number of symbols equal to $1$.

\begin{definition}[Height-one pair]
\label{def:heightone}
Two words $H,L\in\{1,2\}^n$ form a height-one pair if
\begin{equation}
 \delta_k
 :=N_1(H_{1:k})-N_1(L_{1:k})
 \in\{0,1\},
 \qquad k=1,\ldots,n.
 \label{eq:heightone}
\end{equation}
\end{definition}

It is useful to plot $k\mapsto\delta_k$, with $\delta_0:=0$.
The height-one condition means that this path never leaves the two
levels $0$ and $1$.  At level $0$, a pair of unequal next letters must
be $(1,2)$ and moves the discrepancy to $1$; at level $1$, unequal next
letters must be $(2,1)$ and return it to $0$.  Common letters leave the
discrepancy unchanged.  Thus a height-one pair can be viewed as a
sequence of excursions between the two levels.

\begin{definition}[Elementary word operations]
Suppose a word has the form $u21v$.  Replacing the adjacent factor
$21$ by $12$ is an \emph{adjacent move}, written
$u21v\to u12v$, and the word $u$ is its \emph{left prefix}.  If a
word has the form $u2v$, replacing that displayed $2$ by $1$ is a
\emph{replacement} $u2v\to u1v$, again occurring at left prefix $u$.
\end{definition}

\begin{example}[The two-level picture]
Let
\[
 H=1212,\qquad L=2121.
\]
Then
\[
\begin{array}{c|ccccc}
 k&0&1&2&3&4\\ \hline
 N_1(H_{1:k})-N_1(L_{1:k})&0&1&0&1&0
\end{array}
\]
so $(H,L)$ is height-one.  Moreover
\[
 2121\longrightarrow1221\longrightarrow1212,
\]
where both arrows are $21\to12$ adjacent moves.  Their left prefixes are
$\varnothing$ and $12$, which form a prefix chain
$\varnothing\pref12$.  The following lemma shows that every height-one
pair admits a decomposition with this nested-prefix feature, even when
many such excursions are present.
\end{example}

\begin{lemma}[Chain decomposition]
\label{lem:chain}
Let $H,L\in\{1,2\}^n$ be a height-one pair.  Then $L$ can be
transformed into $H$ by a finite sequence of adjacent moves
\begin{equation}
 21\longrightarrow12
 \label{eq:bubble}
\end{equation}
and, if $\delta_n=1$, exactly one replacement
\begin{equation}
 2\longrightarrow1.
 \label{eq:replacement}
\end{equation}
The left prefixes at which all operations occur form a chain under
$\pref$.  If the replacement \eqref{eq:replacement} occurs, it occurs
at a deepest prefix in that chain.
\end{lemma}

\begin{proof}
We prove by induction on $n$ that every height-one pair of words of
length $n$ admits a transformation satisfying all the conclusions of
the lemma.  For $n=0$, the empty sequence of operations has the required
properties.

Suppose $n\ge1$ and the claim holds for shorter pairs.  If
$H_1=L_1=s$, write $H=sH'$ and $L=sL'$.  The pair $(H',L')$ is
height-one because deleting the common first symbol does not change any
prefix discrepancy.  Apply the induction hypothesis to transform $L'$
into $H'$, and perform the same operations after the common initial
symbol $s$.  An operation with left prefix $u$ in the shorter pair then
has left prefix $su$ in the original pair.  Thus
$u\pref v$ if and only if $su\pref sv$, so the prefixes still form a
chain.  Moreover, prepending $s$ increases every prefix length by one;
hence a replacement prefix of maximal length remains of maximal length.

It remains to consider $H_1\ne L_1$.  The height-one condition forces
$H_1=1$ and $L_1=2$.  For $n=1$ the unique operation is the replacement
$2\to1$ at the empty prefix.  For $n\ge2$, the pair of second symbols
cannot be $(1,2)$, since that would give $\delta_2=2$.  Thus there are
three cases.

\begin{enumerate}[label=(\alph*),leftmargin=*]
\item If
\[
 H=11H',\qquad L=21L',
\]
first apply the adjacent move $21L'\to12L'$ to the first two symbols.
Its left prefix is $\varnothing$.  It remains to transform $2L'$ into
$1H'$ after the common initial symbol $1$.  The pair $(1H',2L')$ is
height-one: its first prefix discrepancy is $1$, and its discrepancy at
prefix length $j\ge2$ equals $\delta_{j+1}$ for the original pair.
By induction, its operation prefixes form a chain.  In the original
words these prefixes have the form $1u$, and
$\varnothing\pref1u$ for every such $u$; hence all operation prefixes
form a chain.  If the recursive transformation contains a replacement,
its prefix has maximal length after the initial $1$ is prepended and is
therefore also maximal in the full chain.

\item If
\[
 H=12H',\qquad L=21L',
\]
first apply the adjacent move $21L'\to12L'$ to the first two symbols;
again its left prefix is $\varnothing$.  Since the discrepancy returns
to $0$ after these two symbols, $(H',L')$ is height-one.  Apply induction
to this pair after the common prefix $12$.  The corresponding operation
prefixes are $12u$, so they form a chain together with
$\varnothing$.  If a replacement occurs, its prefix remains of maximal
length after $12$ is prepended.

\item If
\[
 H=12H',\qquad L=22L',
\]
the pair $(1H',2L')$ is height-one: its first prefix discrepancy is $1$,
and its discrepancy at prefix length $j\ge2$ equals $\delta_{j+1}$ for
the original pair.  Apply induction to this pair after the common
initial symbol $2$, thereby transforming $22L'$ into $21H'$.  Its
operation prefixes become words of the form $2u$.  Finally apply the
adjacent move $21H'\to12H'$ to the first two symbols; this last move has
left prefix $\varnothing$.  Thus the full collection of prefixes is a
chain.  If the recursive transformation contains a replacement, its
prefix remains of maximal length after the initial $2$ is prepended.
\end{enumerate}

The induction and the three cases therefore produce a transformation
from $L$ to $H$ using only the allowed operations, with all left prefixes
forming a chain and any replacement occurring at a deepest prefix.

Let $N_{\mathrm{rep}}$ be the number of replacements in the resulting
transformation.
Each adjacent move preserves the number of $1$'s because both $21$ and
$12$ contain one symbol $1$, whereas each replacement $2\to1$ increases
that number by one.  Since the transformation begins at $L$ and ends at
$H$,
\[
 N_{\mathrm{rep}}=N_1(H)-N_1(L)=\delta_n.
\]
The height-one condition gives $\delta_n\in\{0,1\}$, so there is exactly
one replacement when $\delta_n=1$ and none when $\delta_n=0$.  This
completes the induction.
\end{proof}

\subsection{Positive \texorpdfstring{$2\times2$}{2 x 2} matrices in projective coordinates}

Throughout the rest of this section, assume the conditions in
\eqref{eq:positive-strict}.  Then the matrices $A_1$ and $A_2$ defined
in \eqref{eq:A1}--\eqref{eq:A2} are entrywise positive and
\begin{equation}
 \det A_i=\alpha_i\Delta>0.
 \label{eq:detAi}
\end{equation}
The following definition introduces the only notions from total
positivity needed here.

\begin{definition}[Total positivity and projective slope]
A $2\times2$ matrix is \emph{strictly totally positive} if all of its
minors are positive.  Writing
\[
 B=\begin{pmatrix}b_{11}&b_{12}\\ b_{21}&b_{22}\end{pmatrix},
\]
this means simply
\[
 b_{11},b_{12},b_{21},b_{22}>0,
 \qquad \det B>0.
\]
For a positive column vector $v=(v_1,v_2)^\top$, define its projective
slope by
\[
 \operatorname{sl}(v):=\frac{v_1}{v_2}.
\]
If $b^{(1)},b^{(2)}$ are the two columns of a strictly totally positive
$B$, then
\[
 \operatorname{sl}(b^{(2)})=\frac{b_{12}}{b_{22}}
 <\frac{b_{11}}{b_{21}}=\operatorname{sl}(b^{(1)}).
\]
We call the interval between these two slopes the \emph{projective
interval} of $B$:
\begin{equation}
 \cI(B):=
 \left[\frac{b_{12}}{b_{22}},\frac{b_{11}}{b_{21}}\right].
 \label{eq:IB}
\end{equation}
For the identity matrix, which is not entrywise positive, we use the
convenient convention
\begin{equation}
 \cI(I):=[0,\infty].
 \label{eq:Iempty}
\end{equation}
\end{definition}

The terminology is projective because multiplying $v$ by a positive
scalar leaves $\operatorname{sl}(v)$ unchanged: the slope records the
ray through $v$, not its magnitude.  No further projective geometry is
needed below.
For background on total positivity and its variation-diminishing
properties, see~\cite{pinkus2010}; for characterisations in terms of
sign non-reversal, see~\cite{choudhury2021}.  We need only
the elementary $2\times2$ geometry just defined.  By
\eqref{eq:detAi}, every nonempty product $A_u$ is strictly
totally positive.

\begin{lemma}[Nested projective intervals]
\label{lem:nesting}
If $u\pref u'$, then
\begin{equation}
 \cI(A_{u'})\subseteq\cI(A_u),
 \label{eq:nesting}
\end{equation}
with the convention \eqref{eq:Iempty}.
\end{lemma}

\begin{proof}
If $u=\varnothing$, then \eqref{eq:nesting} follows from
$\cI(I)=[0,\infty]$.  Otherwise write $u'=uz$.  The conclusion is an
equality if $z=\varnothing$, so suppose $z\ne\varnothing$.  Write
\[
 A_u=\bigl(a^{(1)}\;a^{(2)}\bigr),
 \qquad
 a^{(1)}=\binom{a_{11}}{a_{21}},
 \qquad
 a^{(2)}=\binom{a_{12}}{a_{22}},
\]
and set
\[
 s_1:=\frac{a_{11}}{a_{21}},
 \qquad
 s_2:=\frac{a_{12}}{a_{22}}.
\]
Since $A_u$ is strictly totally positive, $s_2<s_1$ and
$\cI(A_u)=[s_2,s_1]$.

Let $A_z=(z_{ij})$.  This matrix is entrywise positive, and the $j$th
column of $A_{u'}=A_uA_z$ is
\[
 a'^{(j)}=z_{1j}a^{(1)}+z_{2j}a^{(2)},
 \qquad j=1,2.
\]
Consequently
\[
 \operatorname{sl}(a'^{(j)})
 =\frac{z_{1j}a_{21}s_1+z_{2j}a_{22}s_2}
        {z_{1j}a_{21}+z_{2j}a_{22}}
 \in[s_2,s_1].
\]
Thus the slopes of both columns of $A_{u'}$ belong to $\cI(A_u)$.
Moreover,
\[
 \det A_{u'}=\det A_u\det A_z>0,
\]
so the second-column slope of $A_{u'}$ is smaller than its first-column
slope.  Hence the interval
between these slopes is contained in $\cI(A_u)$, which is precisely
\eqref{eq:nesting}.
\end{proof}

\subsection{Local swaps and projective separators}

Before computing the local swap, we name the certificate that will be
used repeatedly.

\begin{definition}[Projective separator]
For a $2\times2$ matrix $X$ and $t>0$, let $\xi_t=(1,-t)$ be a row
vector.  We say
that $t$ \emph{separates} $X$ if
\[
 \xi_tX\le0
\]
componentwise.
\end{definition}

If $c>0$ and $d=Xc=(d_1,d_2)^\top$, separation implies
$d_1-td_2\le0$.  In particular,
\[
 d_2<0\quad\Longrightarrow\quad d_1<0.
\]
Thus a separator is exactly the kind of certificate that rules out the
endpoint sign pattern $d_2<0<d_1$ corresponding to an upward crossing
of an affine word-cost difference.
We next show that every adjacent move $21\to12$ has a whole interval of
such separators.

Set
\begin{equation}
 \beta_1:=1-\alpha_1,
 \qquad
 \beta_2:=1-\alpha_2.
 \label{eq:beta}
\end{equation}
These are the one-search detection probabilities at the two sites, and
both are strictly positive under \eqref{eq:positive-strict}.
Directly from \eqref{eq:Rw},
\begin{equation}
 K:=R_{12}-R_{21}
 =\begin{pmatrix}0&-\beta_1\\\beta_2&0\end{pmatrix}.
 \label{eq:K}
\end{equation}
The fixed sign pattern of $K$ is the local source of the single-crossing
property:
a row vector with signs $(+,-)$ becomes componentwise nonpositive after
multiplication by $K$.  The next lemma shows that the difference between
the two future update products factors through the same $K$, with an
entrywise positive matrix on the right.

\begin{lemma}[Commutator factorisation]
\label{lem:commutator}
One has
\begin{equation}
 A_1A_2-A_2A_1=K\Gamma,
 \label{eq:commutator}
\end{equation}
where
\begin{equation}
 \Gamma=(1-\alpha_1\alpha_2)
 \begin{pmatrix}
 \dfrac{a(1-b)}{\beta_2}&\dfrac{(1-a)(1-b)}{\beta_2}\\[1.2ex]
 \dfrac{(1-a)(1-b)}{\beta_1}&\dfrac{b(1-a)}{\beta_1}
 \end{pmatrix}.
 \label{eq:Gamma}
\end{equation}
In particular, $\Gamma$ is entrywise positive.
\end{lemma}

\begin{proof}
The factorisation follows by a direct entrywise calculation.  From
\eqref{eq:A1}--\eqref{eq:A2}, the four entries of the commutator are
\begin{align*}
 (A_1A_2-A_2A_1)_{11}
 &=\alpha_1\bigl[a^2+\alpha_2(1-a)(1-b)\bigr]
   -\bigl[\alpha_1a^2+(1-a)(1-b)\bigr]\\
 &=-(1-\alpha_1\alpha_2)(1-a)(1-b),\\
 (A_1A_2-A_2A_1)_{12}
 &=\alpha_1(1-a)(a+\alpha_2b)
   -(1-a)(\alpha_1a+b)\\
 &=-(1-\alpha_1\alpha_2)b(1-a),\\
 (A_1A_2-A_2A_1)_{21}
 &=(1-b)(a+\alpha_2b)
   -\alpha_2(1-b)(\alpha_1a+b)\\
 &=(1-\alpha_1\alpha_2)a(1-b),\\
 (A_1A_2-A_2A_1)_{22}
 &=(1-a)(1-b)+\alpha_2b^2
   -\alpha_2\bigl[\alpha_1(1-a)(1-b)+b^2\bigr]\\
 &=(1-\alpha_1\alpha_2)(1-a)(1-b).
\end{align*}
Consequently
\begin{equation}
 A_1A_2-A_2A_1
 =(1-\alpha_1\alpha_2)
 \begin{pmatrix}
 -(1-a)(1-b)&-b(1-a)\\
 a(1-b)&(1-a)(1-b)
 \end{pmatrix}.
 \label{eq:commutator-expanded}
\end{equation}
On the other hand, using \eqref{eq:beta} and \eqref{eq:Gamma},
\[
 K\Gamma
 =(1-\alpha_1\alpha_2)
 \begin{pmatrix}
 0&-\beta_1\\ \beta_2&0
 \end{pmatrix}
 \begin{pmatrix}
 \dfrac{a(1-b)}{\beta_2}&\dfrac{(1-a)(1-b)}{\beta_2}\\[1.2ex]
 \dfrac{(1-a)(1-b)}{\beta_1}&\dfrac{b(1-a)}{\beta_1}
 \end{pmatrix},
\]
which is exactly the matrix in \eqref{eq:commutator-expanded}.  This proves
\eqref{eq:commutator}.  Every entry of $\Gamma$ is positive because
$0<a,b<1$ and $0<\alpha_1,\alpha_2<1$ by
\eqref{eq:positive-strict}.
\end{proof}

\begin{lemma}[Adjacent-move separator]
\label{lem:bubble-separator}
For any words $u,v$,
\begin{equation}
 R_{u12v}-R_{u21v}
 =A_uK(I+\Gamma R_v).
 \label{eq:bubble-factor}
\end{equation}
If $t\in\cI(A_u)\cap(0,\infty)$ and $\xi_t:=(1,-t)$, then
\begin{equation}
 \xi_t(R_{u12v}-R_{u21v})\le0
 \label{eq:bubble-halfspace}
\end{equation}
componentwise.
\end{lemma}

\begin{proof}
Using \eqref{eq:concat},
\begin{align*}
 R_{u12v}-R_{u21v}
 &=A_u\bigl(R_{12}-R_{21}
 +(A_1A_2-A_2A_1)R_v\bigr)\\
 &=A_uK(I+\Gamma R_v),
\end{align*}
which proves \eqref{eq:bubble-factor}.

First suppose $u\ne\varnothing$ and write
$A_u=\bigl(\begin{smallmatrix}a_{11}&a_{12}\\a_{21}&a_{22}\end{smallmatrix}\bigr)$.  For
$t\in\cI(A_u)$,
\[
 \xi_tA_u=(a_{11}-ta_{21},\;a_{12}-ta_{22})
 \in[0,\infty)\times(-\infty,0].
\]
Consequently
\begin{equation}
 \xi_tA_uK
 =\bigl(\beta_2(a_{12}-ta_{22}),-\beta_1(a_{11}-ta_{21})\bigr)\le0.
 \label{eq:xiAK}
\end{equation}
Since $I+\Gamma R_v$ is nonnegative, right multiplication preserves the
componentwise inequality.  If $u=\varnothing$, then for every $t>0$
\[
 \xi_tK=(-t\beta_2,-\beta_1)\le0,
\]
so the same conclusion holds under convention \eqref{eq:Iempty}.
\end{proof}

\subsection{The unmatched replacement}

A height-one pair with $\delta_n=1$ requires one $2\to1$ replacement.
It also admits a separator compatible with the projective interval.
For a positive matrix $B$, define the projective slope of its row-sum
vector by
\begin{equation}
 \sigma(B):=\frac{(B\1)_1}{(B\1)_2}.
 \label{eq:sigma}
\end{equation}
We also set $\sigma(I)=1$.

\begin{lemma}[Replacement separator]
\label{lem:replacement-separator}
For every nonempty positive $B$ with positive determinant,
\begin{equation}
 \sigma(B)\in\operatorname{int}\cI(B).
 \label{eq:sigma-in-I}
\end{equation}
Moreover, for any words $u,v$,
\begin{equation}
 (1,-\sigma(A_u))(R_{u1v}-R_{u2v})\le0
 \label{eq:replacement-halfspace}
\end{equation}
componentwise, with the same statement valid for $u=\varnothing$.
\end{lemma}

\begin{proof}
The two positive columns of $B$ have distinct projective slopes
because $\det B>0$.  Since $B\1$ is their positive sum, its
top-to-bottom ratio lies strictly between those two slopes.  This
proves \eqref{eq:sigma-in-I}.

From \eqref{eq:concat},
\begin{equation}
 R_{u1v}-R_{u2v}
 =A_u\bigl(E_1-E_2+(A_1-A_2)R_v\bigr).
 \label{eq:replacement-factor}
\end{equation}
Also
\begin{equation}
 A_1-A_2
 =\begin{pmatrix}-\beta_1&0\\0&\beta_2\end{pmatrix}M.
 \label{eq:A-diff}
\end{equation}
Let $\xi=(1,-\sigma(A_u))$.  By definition,
$\xi A_u\1=0$.  If $u\ne\varnothing$, relation
\eqref{eq:sigma-in-I} shows that the two components of $\xi A_u$ have
signs $+,-$; since their sum is zero,
\begin{equation}
 \xi A_u=\kappa(1,-1)
 \label{eq:kappa}
\end{equation}
for some $\kappa>0$.  The same identity holds with $\kappa=1$ when
$u=\varnothing$.  Now
\[
 (1,-1)(E_1-E_2)=0
\]
and, using \eqref{eq:A-diff},
\[
 (1,-1)(A_1-A_2)=-(\beta_1,\beta_2)M\le0.
\]
Since $R_v\ge0$, substitution into \eqref{eq:replacement-factor}
gives \eqref{eq:replacement-halfspace}.
\end{proof}

We can now aggregate arbitrarily many adjacent moves without ever needing
their individual cost coordinates to have fixed signs.  This is where
the two earlier ingredients meet: Lemma~\ref{lem:chain} makes the
operation prefixes a chain, while Lemma~\ref{lem:nesting} makes their
projective intervals a nested family.  The deepest operation therefore
supplies a separator that is valid for all earlier operations.

\begin{theorem}[Common projective separator]
\label{thm:common-separator}
Let $H,L\in\{1,2\}^n$ be a height-one pair.  Then there exists
$t_*>0$ such that
\begin{equation}
 (1,-t_*)(R_H-R_L)\le0
 \label{eq:common-separator}
\end{equation}
componentwise.
\end{theorem}

\begin{proof}
If $H=L$, then $R_H-R_L=0$ and the result is immediate, for example
with $t_*=1$.  Otherwise the chain decomposition from
Lemma~\ref{lem:chain} contains at least one operation.  Let $u_*$ be a
deepest operation prefix and put
\[
 t_*:=\sigma(A_{u_*}).
\]
For every adjacent move occurring at prefix $u$, chain comparability gives
$u\pref u_*$.  Lemma~\ref{lem:nesting} and
\eqref{eq:sigma-in-I} imply
\[
 t_*\in\cI(A_{u_*})\subseteq\cI(A_u).
\]
Hence Lemma~\ref{lem:bubble-separator} shows that $t_*$ separates the
matrix increment of every adjacent move.  If a replacement is present,
Lemma~\ref{lem:chain} places it at a deepest prefix, and
Lemma~\ref{lem:replacement-separator} gives the same separation inequality for
its matrix increment.
Telescoping the sequence of word transformations from $L$ to $H$ gives
\eqref{eq:common-separator}.
\end{proof}

\begin{corollary}[Height-one single crossing]
\label{cor:heightone-singlecross}
Let $H,L$ be a height-one pair and let
\begin{equation}
 d:=(R_H-R_L)c=\binom{d_1}{d_2}.
 \label{eq:d}
\end{equation}
Then
\begin{equation}
 d_2<0<d_1
 \label{eq:reverse-endpoint}
\end{equation}
is impossible.  Equivalently, the affine function
\begin{equation}
 J_H(p)-J_L(p)=pd_1+(1-p)d_2
 \label{eq:affine-diff}
\end{equation}
cannot cross zero from negative to positive as $p$ increases.
\end{corollary}

\begin{proof}
Theorem~\ref{thm:common-separator} gives
$d_1-t_*d_2\le0$.  If $d_2<0$, then
$d_1\le t_*d_2<0$, ruling out \eqref{eq:reverse-endpoint}.  An affine
function has an upward interior zero crossing if and only if its values
at $0$ and $1$ have the signs in \eqref{eq:reverse-endpoint}.
\end{proof}

\section{Bellman trajectories and finite-horizon thresholds}
\label{sec:finite}

We now explain why the apparently special height-one condition is
forced by the Bellman recursion.  The key is that, in the positive-determinant
regime, the movement step preserves the order of beliefs and contracts
their distance in log odds.  Two branches that begin with opposite
searches can therefore alternate which one has the larger belief, but
they can never get more than one unmatched search ahead.

\subsection{Log-odds contraction}

Assume again the conditions in \eqref{eq:positive-strict}.
For $p\in(0,1)$ set
\begin{equation}
 \eta:=\log\frac{p}{1-p},
 \qquad
 \ell_i:=-\log\alpha_i>0,
 \qquad
 \Lambda:=\ell_1+\ell_2.
 \label{eq:logodds}
\end{equation}
After an unsuccessful search, before target movement, the log odds are
shifted by
\begin{equation}
 \eta\mapsto\eta-\ell_1
 \quad\text{under search 1},
 \qquad
 \eta\mapsto\eta+\ell_2
 \quad\text{under search 2}.
 \label{eq:search-shift}
\end{equation}
To see the effect of target movement in these coordinates, let $z$ be
the probability that the target is at site~1 immediately before
movement.  After applying $M$, the probabilities that it is at
sites~1 and~2 are
\[
 az+(1-b)(1-z)
 \quad\text{and}\quad
 (1-a)z+b(1-z),
\]
respectively.  If $z/(1-z)=e^\eta$, their new log-odds ratio is therefore
\begin{equation}
 \Phi(\eta)
 :=\log\frac{ae^\eta+1-b}{(1-a)e^\eta+b}.
 \label{eq:Phi}
\end{equation}
Thus $\Phi$ is simply the Markov movement step written in log-odds
coordinates.

\begin{lemma}[Strict log-odds contraction]
\label{lem:Phi}
If $0<a,b<1$ and $\Delta>0$, then
\begin{equation}
 0<\Phi'(\eta)<1
 \qquad\text{for every }\eta\in\R.
 \label{eq:Phi-derivative-bound}
\end{equation}
\end{lemma}

\begin{proof}
Put $x=e^\eta>0$.  Differentiating the two logarithms in
\eqref{eq:Phi} gives
\[
 \Phi'(\eta)
 =
 \frac{ax}{ax+1-b}
 -
 \frac{(1-a)x}{(1-a)x+b}.
\]
Putting the two fractions over a common denominator yields
\begin{equation}
 \Phi'(\eta)
 =\frac{(a+b-1)x}
 {(ax+1-b)((1-a)x+b)}
 =\frac{\Delta x}
 {(ax+1-b)((1-a)x+b)}.
 \label{eq:Phi-prime}
\end{equation}
Under \eqref{eq:positive-strict} the numerator is positive, so
$\Phi'(\eta)>0$: movement preserves the order of beliefs.  To prove
that it is also a strict contraction, subtract the numerator
$\Delta x$ from the positive denominator:
\begin{align}
 &(ax+1-b)((1-a)x+b)-\Delta x \notag\\
 &\qquad
 =a(1-a)x^2+2(1-a)(1-b)x+b(1-b)>0.
 \label{eq:contraction-slack}
\end{align}
Hence $\Phi'(\eta)<1$.
\end{proof}

\subsection{The Bellman branch pair is height-one}

To connect the word analysis with the Bellman recursion, compare two
trajectories that differ only in their forced first action.  After that
action, both use the same monotone stage selector whenever the same
number of searches remains.  The next lemma shows that the resulting
action words necessarily form a height-one pair.

\begin{lemma}[Height-one Bellman branches]
\label{lem:bellman-heightone}
Assume \eqref{eq:positive-strict}.  Fix $n\ge0$ and, for
$m=1,\ldots,n$, let $\mu_m$ be a monotone stage selector: it selects
site~2 below a threshold and site~1 above it, with ties resolved
monotonically.  Starting from a common belief $p_0\in(0,1)$, construct
two trajectories conditional on every search being unsuccessful, and
denote their length-$(n+1)$ action words by $H$ and $L$.  Set $H_1=1$
and $L_1=2$; at each later epoch with $m$ searches remaining, use the
same selector $\mu_m$ on both trajectories.  Then $H$ and $L$ form a
height-one pair.
\end{lemma}

\begin{proof}
Let $N:=|H|=|L|$.  For $k=0,\ldots,N$, define the prefix-count
discrepancy
\[
 \delta_k:=N_1(H_{1:k})-N_1(L_{1:k}),
 \qquad \delta_0:=0.
\]
For $k=1,\ldots,N$, conditional on the first $k$ searches having
failed, let $\eta_{H,k}$ and $\eta_{L,k}$ be the log odds after the
target movement following the $k$th search along the two trajectories.
Set
\[
 g_k:=\eta_{L,k}-\eta_{H,k}.
\]
Thus $g_k>0$ means that the $L$-trajectory has the larger belief at the
next search epoch.

The forced first actions are $H_1=1$ and $L_1=2$, so $\delta_1=1$.
Before the first movement, \eqref{eq:search-shift} gives a log-odds gap
of
\[
 (\eta_0+\ell_2)-(\eta_0-\ell_1)=\Lambda,
 \]
where $\eta_0=\log(p_0/(1-p_0))$.  Lemma~\ref{lem:Phi} shows that the
common movement step preserves the sign of this gap and strictly
reduces its magnitude.  Hence
\[
 0<g_1<\Lambda,
 \qquad \delta_1=1.
\]

We now prove inductively that, for every $k=1,\ldots,N$,
\begin{align*}
 \delta_k=1&\quad\Longrightarrow\quad 0<g_k<\Lambda,\\
 \delta_k=0&\quad\Longrightarrow\quad -\Lambda<g_k<0.
\end{align*}
Suppose the assertion holds at $k<N$.  At epoch $k+1$, both
trajectories use the same monotone stage selector.  If they choose the
same site, \eqref{eq:search-shift} applies the same shift to both log
odds.  Thus the pre-movement gap remains $g_k$ and
$\delta_{k+1}=\delta_k$.  Lemma~\ref{lem:Phi} then preserves the sign of
the gap and reduces its magnitude, proving the assertion at $k+1$.

It remains to consider different actions.  If $\delta_k=1$, then
$g_k>0$, so the $L$-trajectory has the higher belief.  Monotonicity of
the common selector forces
\[
 H_{k+1}=2,
 \qquad
 L_{k+1}=1.
\]
Consequently $\delta_{k+1}=0$.  Immediately after these searches and
before movement, the new log-odds gap is
\[
 (\eta_{L,k}-\ell_1)-(\eta_{H,k}+\ell_2)
 =g_k-\Lambda\in(-\Lambda,0).
\]
Lemma~\ref{lem:Phi} preserves its negative sign and reduces its
magnitude, so $-\Lambda<g_{k+1}<0$.

If $\delta_k=0$, then $g_k<0$, so the $H$-trajectory has the higher
belief.  Different actions therefore require
\[
 H_{k+1}=1,
 \qquad
 L_{k+1}=2,
\]
and hence $\delta_{k+1}=1$.  The pre-movement gap is now
\[
 (\eta_{L,k}+\ell_2)-(\eta_{H,k}-\ell_1)
 =g_k+\Lambda\in(0,\Lambda).
\]
After movement, Lemma~\ref{lem:Phi} gives
$0<g_{k+1}<\Lambda$.  This completes the induction.  In particular,
$\delta_k\in\{0,1\}$ for every $k$, so $(H,L)$ is a height-one pair by
Definition~\ref{def:heightone}.
\end{proof}

The stage selectors in the lemma need not be stationary: $\mu_m$ may
depend on the number $m$ of searches remaining.  The proof requires
only that, whenever $m$ searches remain, both trajectories use the same
stage selector $\mu_m$.  This is precisely the comparison used in the
finite-horizon induction below.

\subsection{Finite-horizon induction}

\begin{theorem}[Finite-horizon threshold property]
\label{thm:finite}
Under \eqref{eq:positive-strict}, for every
$n\ge1$ the first-stage minimiser for $V_n$ admits a monotone threshold
stage selector.  Choosing such a selector $\mu_m$ whenever $m$ searches
remain gives an optimal finite-horizon policy whose decision at every
stage is threshold in the current belief.  Equivalently, the first-stage
action differential satisfies
\begin{equation}
 p<p',
 \qquad D_n(p)<0
 \quad\Longrightarrow\quad
 D_n(p')\le0.
 \label{eq:finite-one-sided}
\end{equation}
\end{theorem}

\begin{proof}
We prove by induction on $n$ that $D_n$ has the one-sided crossing
property \eqref{eq:finite-one-sided} and that its minimising action can
be chosen by a monotone threshold stage selector.  For $n=1$,
\[
 D_1(p)=C_1-C_2
\]
is constant, so both conclusions are immediate.

Assume both conclusions hold through horizon $n$, and fix a monotone
optimal stage selector for each of these horizons.  We first prove the
one-sided crossing property for $D_{n+1}$.  Suppose, to the contrary,
that there exist $p<p'$ such that
\begin{equation}
 D_{n+1}(p)<0<D_{n+1}(p').
 \label{eq:contradict-signs}
\end{equation}
By Proposition~\ref{prop:open-loop}, the two branch values can be
written as
\[
 Q_{n+1,i}(x)
 =\min_{\substack{w\in\{1,2\}^{n+1}\\w_1=i}}J_w(x),
 \qquad i=1,2.
\]
Each is therefore a finite minimum of continuous affine functions, so
$D_{n+1}=Q_{n+1,1}-Q_{n+1,2}$ is continuous.  Define
\begin{equation}
 \widehat p:=\inf\{x\in[p,p']:D_{n+1}(x)>0\}.
 \label{eq:first-positive-point}
\end{equation}
The strict signs in \eqref{eq:contradict-signs} and continuity imply
$p<\widehat p<p'$ and $D_{n+1}(\widehat p)=0$.  By the definition of
the infimum, there is a sequence $p_j\downarrow\widehat p$ such that
$D_{n+1}(p_j)>0$.

At each $p_j$, force first action~1 and then use the fixed monotone
optimal selectors for horizons $n,n-1,\ldots,1$; let $H_j$ be the
resulting word.  Define $L_j$ similarly after forcing first action~2.
These words attain the two Bellman branch values at $p_j$.  There are
only finitely many pairs of words of length $n+1$.  Passing to a
subsequence and relabelling that subsequence again as $(p_j)$, we may
therefore assume that, for every $j$,
\[
 H_j=H,
 \qquad
 L_j=L.
\]
Lemma~\ref{lem:bellman-heightone} shows that $(H,L)$ is a height-one
pair.  Let
\[
 f(x):=J_H(x)-J_L(x).
\]
It is affine by Lemma~\ref{lem:word-cost}.  At every point of the
subsequence,
\begin{equation}
 f(p_j)=D_{n+1}(p_j)>0.
 \label{eq:fpj}
\end{equation}
Continuity of $J_H$ and of the first branch value gives
\[
 J_H(\widehat p)
 =\lim_jJ_H(p_j)
 =\lim_jQ_{n+1,1}(p_j)
 =Q_{n+1,1}(\widehat p),
\]
and the same argument gives
$J_L(\widehat p)=Q_{n+1,2}(\widehat p)$.  Hence
\begin{equation}
 f(\widehat p)=D_{n+1}(\widehat p)=0.
 \label{eq:fphat0}
\end{equation}
Because $f$ is affine, $p_j>\widehat p$, and
$f(p_j)>f(\widehat p)=0$, the function $f$ is strictly increasing.
Since $0<\widehat p<1$, it follows that
\[
 f(0)<0<f(1).
\]
Write $d=(R_H-R_L)c=(d_1,d_2)^\top$.  By
\eqref{eq:affine-diff},
\[
 f(x)=xd_1+(1-x)d_2,
 \qquad
 f(0)=d_2,
 \qquad
 f(1)=d_1.
\]
Thus $f(0)<0<f(1)$ is exactly the forbidden sign pattern
$d_2<0<d_1$ in Corollary~\ref{cor:heightone-singlecross}.  This
contradiction proves \eqref{eq:finite-one-sided} for $D_{n+1}$.

It remains only to select a threshold minimiser.  Set
\[
 p^*:=\sup\{x\in[0,1]:D_{n+1}(x)>0\},
\]
with the convention $\sup\varnothing=0$.  If $x<p^*$ and
$D_{n+1}(x)<0$, the definition of the supremum provides some
$y\in(x,p^*]$ with $D_{n+1}(y)>0$, contradicting
\eqref{eq:finite-one-sided}.  Hence $D_{n+1}(x)\ge0$ for $x<p^*$.
For $x>p^*$ the definition of $p^*$ gives $D_{n+1}(x)\le0$.
Because $D_{n+1}=Q_{n+1,1}-Q_{n+1,2}$, site~2 is therefore a minimiser
below $p^*$ and site~1 is a minimiser above $p^*$.  At $p^*$ choose a
minimising action, with either choice allowed in case of a tie.  This
produces an optimal monotone stage selector for horizon $n+1$ and closes
the induction.  The conventions also cover the degenerate cases
$p^*=0$ and $p^*=1$.
\end{proof}

\section{Infinite horizon}
\label{sec:infinite}

We now pass from the truncated values $V_n$ to the infinite-horizon
infimum $V$ defined in Section~\ref{sec:model}.  A simple reference
policy first shows that $V$ is finite and gives a uniform bound that is
independent of $a$ and $b$, which is also useful at transition-matrix
boundaries.

Set
\begin{equation}
 \beta_{\min}:=\min\{1-\alpha_1,1-\alpha_2\}>0,
 \qquad
 C_{\min}:=\min\{C_1,C_2\},
 \qquad
 C_{\max}:=\max\{C_1,C_2\}.
 \label{eq:beta-min}
\end{equation}
Consider the policy that at every epoch searches a site whose current
posterior probability is at least $1/2$.  Conditional on having reached
that epoch, its probability of detection is at least $\beta_{\min}/2$.
Hence its expected number of searches is at most $2/\beta_{\min}$ and its
expected total cost is at most
\begin{equation}
 V_{\mathrm{ref}}:=\frac{2C_{\max}}{\beta_{\min}}.
 \label{eq:Vref}
\end{equation}
In particular,
\begin{equation}
 V(p)\le V_{\mathrm{ref}},
 \qquad
 V_n(p)\le V_{\mathrm{ref}}
 \label{eq:Vbound}
\end{equation}
for every $p\in[0,1]$, every $n\ge0$, and every
$a,b\in[0,1]$.

\begin{lemma}[Uniform truncation bound]
\label{lem:truncation}
For every $n\ge1$,
\begin{equation}
 0\le V(p)-V_n(p)
 \le\frac{V_{\mathrm{ref}}^2}{nC_{\min}}
 \qquad\text{for all }p\in[0,1].
 \label{eq:truncation}
\end{equation}
Consequently $V_n\to V$ uniformly.
\end{lemma}

\begin{proof}
The first inequality is immediate because the truncated problem charges
only the search costs incurred until detection or the $n$th search and assigns
zero terminal cost if the target is still undetected at truncation.  Let $\pi_n$ be an optimal
$n$-search policy, and let $\tau_n$ denote the number of searches it
actually performs before detection or truncation.  Since every search
costs at least $C_{\min}$,
\[
 C_{\min}\EE_{\pi_n}\tau_n
 \le V_n(p)\le V_{\mathrm{ref}}.
\]
Let $\mathcal S_n$ be the event that the target remains undetected after all
$n$ searches.  On $\mathcal S_n$ we have $\tau_n=n$, hence Markov's inequality in
this elementary form gives
\begin{equation}
 \PP_{\pi_n}(\mathcal S_n)
 \le\frac{\EE_{\pi_n}\tau_n}{n}
 \le\frac{V_{\mathrm{ref}}}{nC_{\min}}.
 \label{eq:survival-n}
\end{equation}
Extend $\pi_n$, conditional on $\mathcal S_n$, by the reference policy described
above.  The conditional continuation cost is at most $V_{\mathrm{ref}}$.  Therefore
\[
 V(p)
 \le V_n(p)+V_{\mathrm{ref}}\PP_{\pi_n}(\mathcal S_n)
 \le V_n(p)+\frac{V_{\mathrm{ref}}^2}{nC_{\min}},
\]
which proves \eqref{eq:truncation}.
\end{proof}

Define the infinite-horizon differential $D$ by \eqref{eq:D}.  Under
\eqref{eq:positive-strict}, \eqref{eq:truncation} gives
\begin{equation}
 \lVert D_{n+1}-D\rVert_\infty
 \le2\lVert V_n-V\rVert_\infty
 \longrightarrow0.
 \label{eq:Duniform}
\end{equation}

\begin{proposition}[Infinite-horizon single crossing]
\label{prop:infinite-singlecross}
Under \eqref{eq:positive-strict},
\begin{equation}
 p<p',
 \qquad D(p)<0
 \quad\Longrightarrow\quad
 D(p')\le0.
 \label{eq:infinite-singlecross}
\end{equation}
Define
\begin{equation}
 p^*:=\sup\{x\in[0,1]:D(x)>0\},
 \qquad \sup\varnothing:=0.
 \label{eq:infinite-threshold}
\end{equation}
Then site~2 minimises the Bellman right-hand side for every $p<p^*$,
and site~1 minimises it for every $p>p^*$.  Choosing either minimising
action at $p^*$ therefore gives a threshold minimising selector.
\end{proposition}

\begin{proof}
If \eqref{eq:infinite-singlecross} failed, there would exist $p<p'$ with
$D(p)<0<D(p')$.  By uniform convergence \eqref{eq:Duniform}, the same
strict inequalities would hold with $D_{n+1}$ in place of $D$ for all
sufficiently large $n$, contradicting Theorem~\ref{thm:finite}.

It remains to verify the claims following
\eqref{eq:infinite-threshold}.  If $x<p^*$ and $D(x)<0$, the definition
of the supremum gives some $y\in(x,p^*]$ such that $D(y)>0$, contrary to
\eqref{eq:infinite-singlecross}.  Hence $D(x)\ge0$ for every $x<p^*$.
If $x>p^*$, the definition of $p^*$ gives $D(x)\le0$.  Since
$D=Q_1-Q_2$, site~2 is a minimiser below $p^*$ and site~1 is a
minimiser above $p^*$.  At $p^*$, choose either action that minimises
the Bellman right-hand side.
\end{proof}

It remains to verify that a stationary selector minimising the Bellman
right-hand side actually attains $V$ in this undiscounted problem.

\begin{proposition}[Verification]
\label{prop:verification}
Let $\mu$ be a measurable stationary selector attaining the minimum in
\eqref{eq:bellman} at every belief.  Then the policy generated by
$\mu$ has expected total search cost $V(p)$ for every initial belief
$p$.
\end{proposition}

\begin{proof}
First we verify the Bellman equation.  For each $p\in[0,1]$ and
$i\in\{1,2\}$,
\[
 \left|\bigl(C_i+q_i(p)V_n(F_i(p))\bigr)
       -\bigl(C_i+q_i(p)V(F_i(p))\bigr)\right|
 \le \lVert V_n-V\rVert_\infty,
\]
because $0\le q_i(p)\le1$.  The difference between the minima of the
two pairs of expressions is no greater than the largest difference
between corresponding expressions.  Therefore
\[
 \lVert\mathcal T V_n-\mathcal T V\rVert_\infty
 \le \lVert V_n-V\rVert_\infty.
\]
Since $V_{n+1}=\mathcal T V_n$ and $V_n\to V$ uniformly by
Lemma~\ref{lem:truncation},
\[
 \lVert V-\mathcal T V\rVert_\infty
 \le \lVert V-V_{n+1}\rVert_\infty
    +\lVert\mathcal T V_n-\mathcal T V\rVert_\infty
 \longrightarrow0.
\]
Thus $V=\mathcal T V$.

Fix an initial belief $p$ and an integer $n\ge1$.  Along the history in
which the first $n$ searches are all unsuccessful, the selector $\mu$
determines a word $w=w_1\cdots w_n$.  Because $\mu$ selects a
minimising action, repeatedly substituting the selected branch of the
Bellman equation gives
\begin{equation}
 V(p)
 =J_w(p)
  +\rho(p)A_w\1\,
 V\!\left(\frac{(\rho(p)A_w)_1}{\rho(p)A_w\1}\right).
 \label{eq:verification-iterate}
\end{equation}
If $\rho(p)A_w\1=0$, the continuation term in this display is
interpreted as zero, consistently with the convention following
\eqref{eq:F1}--\eqref{eq:F2}.
Here $J_w(p)$ is the expected cost incurred until detection or the
$n$th search, as established in Lemma~\ref{lem:word-cost}, and
$\rho(p)A_w\1$ is the probability that all $n$ searches are
unsuccessful.  Conditional on that event, the argument of $V$ in the
second term is the probability that the target is at site~1 after the
$n$th movement.  Thus the second term is precisely the expected
continuation cost after the first $n$ searches.

The probability of reaching any one of the first $n$ searches is at
least the probability $\rho(p)A_w\1$ of surviving all of them.  Since
each search costs at least $C_{\min}$, Lemma~\ref{lem:word-cost} gives
\[
 J_w(p)\ge nC_{\min}\rho(p)A_w\1.
\]
The continuation term in \eqref{eq:verification-iterate} is
nonnegative, so $J_w(p)\le V(p)\le V_{\mathrm{ref}}$.  Consequently,
\[
 \rho(p)A_w\1
 \le\frac{V_{\mathrm{ref}}}{nC_{\min}}.
\]
Using $V\le V_{\mathrm{ref}}$ in the continuation term of
\eqref{eq:verification-iterate} now gives
\[
 0\le V(p)-J_w(p)
 \le\frac{V_{\mathrm{ref}}^2}{nC_{\min}},
\]
which tends to zero as $n\to\infty$.

As $n$ increases, these words are consistent prefixes of the sequence
of searches generated by $\mu$ along successive unsuccessful outcomes.
The corresponding accumulated search costs increase almost surely to
the total search cost under $\mu$, and their expectations are
$J_w(p)$.  The monotone convergence theorem therefore shows that the
expected total cost under $\mu$ is the limit of $J_w(p)$, which the
preceding bound identifies as $V(p)$.
\end{proof}

Propositions~\ref{prop:infinite-singlecross} and~\ref{prop:verification}
prove Theorem~\ref{thm:positive} in the strict interior.

\section{The nonpositive-determinant regimes}
\label{sec:negative}

MacPhee and Jordan~\cite[Sections~4--5]{macphee1995} established the
negative-determinant regime by a three-case cobweb construction, later
developed further in Jordan's
thesis~\cite[Sections~4.4--4.5]{jordan1997}.  Since the
previously unresolved contribution of the present paper is the
positive-determinant regime, we state the negative-determinant theorem
here.

For the strict negative-determinant regime, assume
\begin{equation}
0<a,b<1,
\qquad
0<\alpha_1,\alpha_2<1,
\qquad
C_1,C_2>0,
\qquad
\Delta=a+b-1<0.
\label{eq:negative-strict}
\end{equation}

\begin{theorem}[Negative-determinant threshold theorem]
\label{thm:negative}
Under \eqref{eq:negative-strict}, the moving-target search problem admits
an optimal deterministic stationary threshold selector.
\end{theorem}

A detailed, self-contained proof is given in
Appendix~\ref{app:negative}, which reorganises the MacPhee--Jordan
cobweb argument.

\subsection{The zero-determinant regime}

When $\Delta=0$, movement completely erases the Bayesian information
created by the preceding miss.

\begin{proposition}[Zero determinant]
\label{prop:zero-determinant}
Suppose
\[
0<a,b<1,
\qquad
0\le\alpha_1,\alpha_2<1,
\qquad
C_1,C_2>0,
\qquad
a+b=1.
\]
Then an optimal deterministic stationary threshold selector exists.
\end{proposition}

\begin{proof}
Because $1-b=a$,
\[
F_1(p)=F_2(p)=a
\qquad(p\in[0,1]).
\]
Thus
\[
D(p)
=C_1-C_2+[q_1(p)-q_2(p)]V(a),
\]
and
\begin{equation}
D'(p)
=-[(1-\alpha_1)+(1-\alpha_2)]V(a)<0.
\label{eq:zero-D-slope}
\end{equation}
Therefore site~1 can become strictly preferable at most once as $p$
increases.  Monotone tie breaking gives a threshold minimiser, and
Proposition~\ref{prop:verification} verifies its optimality.
\end{proof}

\section{Boundary parameters and proof of the main theorem}
\label{sec:boundary}

Sections~\ref{sec:separator}--\ref{sec:infinite} prove the strict
positive-determinant theorem.  Section~\ref{sec:negative} states the
strict negative-determinant theorem and proves the interior
zero-determinant case.  We now close the remaining transition and
proper-detection boundaries.

To make the dependence on the parameter vector explicit, write
$V_\theta$ for the infinite-horizon value $V$ defined in
Section~\ref{sec:model}, $V_{n,\theta}$ for the truncated value $V_n$
defined by \eqref{eq:finite-bellman}, and $D_\theta$ for the
infinite-horizon action differential $D$ defined by \eqref{eq:D}.

\begin{lemma}[Local uniform continuity of the value]
\label{lem:parameter-continuity}
Fix a parameter vector $\theta_0$ with $C_1,C_2>0$ and
$\alpha_1,\alpha_2<1$.  There is a neighbourhood $\mathcal N$ of
$\theta_0$ and a constant $K<\infty$ such that
\begin{equation}
\sup_{\theta\in\mathcal N}\sup_{p\in[0,1]}
|V_\theta(p)-V_{n,\theta}(p)|
\le\frac{K}{n}.
\label{eq:parameter-local-uniform}
\end{equation}
Consequently $V_\theta(p)$ is jointly continuous in $(\theta,p)$ at
every parameter vector with $\alpha_1,\alpha_2<1$.
\end{lemma}

\begin{proof}
Choose the neighbourhood so that, throughout $\mathcal N$,
\[
\max_i\alpha_i\le\bar\alpha<1,
\qquad
0<\underline C\le C_i\le\overline C<\infty.
\]
The reference policy from Section~\ref{sec:infinite} then gives the
common bound
\[
V_\theta(p),\ V_{n,\theta}(p)
\le
\overline V:=\frac{2\overline C}{1-\bar\alpha}.
\]
Repeating the proof of Lemma~\ref{lem:truncation} with
$\overline V$ and $\underline C$ yields
\[
0\le V_\theta(p)-V_{n,\theta}(p)
\le
\frac{\overline V^{\,2}}{n\underline C},
\]
uniformly on $\mathcal N\times[0,1]$.

For fixed $n$, Proposition~\ref{prop:open-loop} writes
$V_{n,\theta}$ as the minimum of finitely many word costs.  Each word
cost is polynomial in the entries of $A_1,A_2$ and linear in the search
costs, hence jointly continuous in $(\theta,p)$.  Thus
$V_{n,\theta}$ is jointly continuous, and the locally uniform limit
$V_\theta$ is jointly continuous as well.
\end{proof}

The normalised belief $F_i$ may be undefined at a boundary point where
$q_i=0$, but the perspective continuation term has a unique continuous
extension.

\begin{lemma}[Continuity of the action differential]
\label{lem:D-parameter-continuity}
On the parameter set $C_i>0$, $\alpha_i<1$, the action differential
$D_\theta(p)$ is jointly continuous in $(\theta,p)$ when
$q_i(p)V_\theta(F_i(p))$ is interpreted as zero at $q_i(p)=0$.
\end{lemma}

\begin{proof}
Away from $q_i=0$, joint continuity follows from
Lemma~\ref{lem:parameter-continuity} and the rational formulas for
$F_i$.  Near $q_i=0$, the local uniform bound in the proof of that
lemma gives
\[
|q_i(p)V_\theta(F_i(p))|
\le q_i(p)\overline V\longrightarrow0,
\]
regardless of the limiting value of the normalised belief.  Hence the
continuation term, and therefore $D_\theta$, extends continuously.
\end{proof}

\begin{lemma}[Closedness of no reverse crossing]
\label{lem:closedness}
Suppose $\theta_m\to\theta$, all overlook probabilities are strictly
below one, and every $D_{\theta_m}$ has no reverse strict crossing:
there are no $p<p'$ with
\[
D_{\theta_m}(p)<0<D_{\theta_m}(p').
\]
Then $D_\theta$ has no reverse strict crossing either.
\end{lemma}

\begin{proof}
If $D_\theta(p)<0<D_\theta(p')$ for some $p<p'$, joint continuity from
Lemma~\ref{lem:D-parameter-continuity} gives the same two strict
inequalities for all sufficiently large $m$, a contradiction.
\end{proof}

\begin{proposition}[Threshold theorem for $\alpha_i<1$]
\label{prop:classical-range}
For every transition matrix $M$ of the form \eqref{eq:M}, every
$0\le\alpha_1,\alpha_2<1$, and every $C_1,C_2>0$, there is an optimal
deterministic stationary threshold selector.
\end{proposition}

\begin{proof}
First suppose
\[
0<a,b<1,
\qquad
0<\alpha_1,\alpha_2<1.
\]
If $\Delta>0$, use Theorem~\ref{thm:positive}; if $\Delta<0$, use
Theorem~\ref{thm:negative}; and if $\Delta=0$, use
Proposition~\ref{prop:zero-determinant}.  Thus every strict-interior
parameter vector has a no-reverse-crossing action differential.

Now let $\theta$ be an arbitrary proper-detection parameter vector,
allowing transition endpoints and $\alpha_i=0$.  Choose strict-interior
vectors $\theta_m\to\theta$.  Each $\theta_m$ belongs to one of the
three determinant regimes already proved, so Lemma~\ref{lem:closedness}
shows that $D_\theta$ has no reverse strict crossing.

Set
\[
p^*:=\sup\{p\in[0,1]:D_\theta(p)>0\},
\]
with $\sup\varnothing=0$.  No reverse strict crossing implies
$D_\theta(p)\ge0$ for $p<p^*$ and $D_\theta(p)\le0$ for $p>p^*$;
zeros can be resolved monotonically.  Hence the Bellman equation has a
threshold minimising selector.

Finally, Lemma~\ref{lem:parameter-continuity} and the nonexpansiveness
of the Bellman operator permit passage to the limit in
$V_{n+1,\theta}=\mathcal T_\theta V_{n,\theta}$, so
$V_\theta=\mathcal T_\theta V_\theta$.
Proposition~\ref{prop:verification}, whose reference-policy bound is
valid for every $\alpha_i<1$, verifies that the threshold selector
attains $V_\theta$.
\end{proof}

\begin{proof}[Proof of Theorem~\ref{thm:ross}]
This is Proposition~\ref{prop:classical-range}.
\end{proof}

We now treat the endpoint $\alpha_i=1$.  These cases are not covered by
the uniform bound \eqref{eq:Vref}, because a search at site $i$ then has
zero probability of detection.  The relevant issue is whether the other,
detectable site can be reached from the ineffective one.

\begin{proposition}[Properness at the overlook endpoints]
\label{prop:endpoint-properness}
The value $V(p)$ is finite for every $p\in[0,1]$ if and only if one of
the following holds:
\begin{enumerate}[label=\textnormal{(\roman*)},leftmargin=2.2em]
\item $\alpha_1<1$ and $\alpha_2<1$;
\item $\alpha_1=1$, $\alpha_2<1$, and $a<1$;
\item $\alpha_2=1$, $\alpha_1<1$, and $b<1$.
\end{enumerate}
\end{proposition}

\begin{proof}
Case (i) follows from the reference policy of Section~\ref{sec:infinite}.
For (ii), always search site~2.  Conditional on the hidden state at the
beginning of any block of two searches, the probability of detection
within that block is at least
\begin{equation}
 \varepsilon_2:=(1-a)(1-\alpha_2)>0.
 \label{eq:endpoint-eps2}
\end{equation}
Indeed, if the target starts the block at site~2 it is detected on the
first search with probability $1-\alpha_2\ge\varepsilon_2$; if it
starts at site~1, it moves to site~2 after the first unsuccessful search
with probability $1-a$ and is then detected on the second search with
probability $1-\alpha_2$.  Thus the probability of surviving $2m$
searches is at most $(1-\varepsilon_2)^m$, and
\[
 V(p)\le \frac{2C_2}{\varepsilon_2}
 \qquad\text{for all }p.
\]
Case (iii) is symmetric, with the policy that always searches site~1 and
\[
 \varepsilon_1:=(1-b)(1-\alpha_1)>0.
\]

Conversely, if $\alpha_1=1$ and $a=1$, a target initially at site~1 is
never detectable, so $V(1)=+\infty$.  If $\alpha_2=1$ and $b=1$, then
$V(0)=+\infty$.  If $\alpha_1=\alpha_2=1$, detection is impossible at
either site and the total cost is infinite under every policy.  These
possibilities exhaust the complement of (i)--(iii).
\end{proof}

\begin{proposition}[Threshold optimality when one search is completely ineffective]
\label{prop:endpoint-threshold}
Suppose exactly one overlook probability equals one.
\begin{enumerate}[label=\textnormal{(\roman*)},leftmargin=2.2em]
\item If $\alpha_1=1$, $\alpha_2<1$, and $a<1$, an optimal deterministic
stationary threshold selector exists.
\item If $\alpha_2=1$, $\alpha_1<1$, and $b<1$, an optimal deterministic
stationary threshold selector exists.
\end{enumerate}
\end{proposition}

\begin{proof}
We prove (i); the other case is symmetric.  Let
$\theta_0=(a,b,1,\alpha_2,C_1,C_2)$ with $a<1$ and
$\alpha_2<1$, and approximate it by parameter vectors $\theta_m$ for
which $\alpha_1^{(m)}<1$ and $\alpha_1^{(m)}\uparrow1$, keeping all
other parameters fixed.  Write $V_\theta$ and $V_{n,\theta}$ for the
infinite- and finite-horizon value functions at parameter vector
$\theta$, and let $D_\theta$ denote the corresponding infinite-horizon
action differential whenever the value is finite.  By
Proposition~\ref{prop:classical-range},
every $\theta_m$ has the no-reverse-crossing property for its
infinite-horizon action differential.

The policy that always searches site~2 has the two-search detection
bound \eqref{eq:endpoint-eps2}, which does not depend on $\alpha_1$.
Hence in a neighbourhood of $\theta_0$ the infinite-horizon values are
uniformly bounded by the common constant
$\overline V_2:=2C_2/\varepsilon_2$.  Repeating the proof of
Lemma~\ref{lem:truncation} with this bound gives
\[
 0\le V_{\theta}(p)-V_{n,\theta}(p)
 \le \frac{\overline V_2^{\,2}}{nC_{\min}}
\]
uniformly in $p$ and in that neighbourhood.  Since the finite-horizon
branch values are continuous in the parameters, the infinite-horizon
branch values and their difference are therefore continuous at
$\theta_0$.  The same uniform convergence permits passage to the limit
in the finite-horizon Bellman recursion, so
$V_{\theta_0}=\mathcal T_{\theta_0}V_{\theta_0}$.  Here
$\mathcal T_{\theta}$ denotes the Bellman operator
\eqref{eq:bellman-operator} with parameter vector $\theta$.

If the endpoint differential had a reverse strict pair
$D_{\theta_0}(p)<0<D_{\theta_0}(p')$ with $p<p'$, the same strict pair
would occur for all sufficiently large $m$, contradicting
Proposition~\ref{prop:classical-range}.  Thus monotone tie breaking gives
a threshold minimiser of the Bellman equation.  The same bound
$\overline V_2$ and the verification argument of
Proposition~\ref{prop:verification} show that this stationary selector
attains $V_{\theta_0}$.
\end{proof}

Proposition~\ref{prop:endpoint-properness} also identifies the endpoint
cases in which no policy has finite expected total cost from every
initial belief.  We therefore state the endpoint extension for those
parameter vectors whose value is finite from every initial belief.

\begin{theorem}[Endpoint extension]
\label{thm:endpoint-extension}
For every parameter vector with $a,b,\alpha_1,\alpha_2\in[0,1]$ and
$C_1,C_2>0$ such that $V(p)<\infty$ for every $p\in[0,1]$, there exist
$p^*\in[0,1]$ and a deterministic stationary selector $\mu$ such that
\[
 p<p^*\Longrightarrow \mu(p)=2,
 \qquad
 p>p^*\Longrightarrow \mu(p)=1.
\]
For every initial belief $p$, the expected total cost under $\mu$ equals
$V(p)$ and is finite.
\end{theorem}

\begin{proof}
If $\alpha_1,\alpha_2<1$, use
Proposition~\ref{prop:classical-range}.  If exactly one overlook probability equals
one, Proposition~\ref{prop:endpoint-properness} implies that the
corresponding ineffective site is not absorbing, so
Proposition~\ref{prop:endpoint-threshold} applies.  By
Proposition~\ref{prop:endpoint-properness}, these cases exhaust the
parameter vectors for which $V(p)<\infty$ for every $p\in[0,1]$.
\end{proof}

\section{Discussion}
\label{sec:discussion}

The proofs use two complementary order mechanisms.  In the
positive-determinant regime, a common projective separator orders
height-one pairs of search words.  In the negative-determinant regime,
the relevant order is instead an ordering of slopes among the affine
Bellman branches associated with the cobweb intervals.  Several features
of both arguments may be useful beyond the present conjecture.

\subsection{Two determinant signs, two order mechanisms}

When $0<a,b<1$ and $0<\alpha_1,\alpha_2<1$, the determinant sign governs
the orientation of every failed-search belief update.  Indeed,
\[
\det A_i=\alpha_i\Delta,
\]
the search update is a translation in log odds, and the derivative of
the subsequent movement map has the sign of $\Delta$.

When $\Delta>0$, order is preserved.  The products $A_u$ retain positive
determinant, their projective cones have a common orientation, and the
Bellman coupling produces height-one word pairs.  The key certificate is
therefore a projective separator established independently of whether
the individual words are optimal.

When $\Delta<0$, every failed-search belief update reverses order.  The
orientation of the products $A_u$ alternates with word length, and the
separator from the positive-determinant regime no longer applies.  The switching points of
the piecewise-affine Bellman solution instead form alternating cobwebs:
starting from one central affine crossing, successive switching points
are obtained by taking inverse images under the failed-search maps.  In
the cobweb verification, the two-step slope comparison of
Lemma~\ref{lem:negative-flattening} shows that the excess cost of the
competing action is decreasing on the low-belief side and increasing on
the high-belief side.  Since both excess costs vanish at the central
switching point, the competing action cannot become cheaper on the
corresponding side.  Thus the negative-determinant certificate is an ordering
of slopes among the affine Bellman branches on the cobweb intervals.

$\Delta=0$ is the boundary between the two geometries.  Both
failed-search maps send every prior to the same post-movement belief, so
the continuation value is common to the two actions and the action
differential is affine.

The transition and detection boundaries require no additional order
mechanism: the strict-interior conclusions pass to the relevant boundary
values by continuity.

\subsection{Why the projective separator is the right invariant}

A first instinct is to compare $R_H-R_L$ entrywise for a pair of words.
That is too strong: after several excursions, individual entries of this
matrix difference can have either sign.  The useful invariant is instead the
existence of one positive projective parameter $t$ for which
\[
 (1,-t)(R_H-R_L)\le0.
\]
The inequality is componentwise.  If the same $t$ gives this inequality
for every increment in a sequence of adjacent moves and a possible
replacement, then summing those inequalities gives it for the total
difference $R_H-R_L$.  The nested intervals $\cI(A_u)$ make it possible
to use the same $t$ for all the operations in the sequence.

For a positive $2\times2$ matrix $B$, the interval $\cI(B)$ is the
interval between the projective images of the two coordinate rays under
$B$.
Because $\det B>0$, their order is preserved.  Multiplication by another
positive matrix sends each new coordinate ray into the cone spanned by
the old two, which explains the nesting geometrically.  This use of
positive matrices through their action on rays is related to Birkhoff's
use of Hilbert's projective metric for
positive operators~\cite{birkhoff1957}, although no projective metric is
used here.  The
algebraic factorisation \eqref{eq:bubble-factor} then turns this
projective order into a cost comparison.

\subsection{Height-one paths as the dynamic admissibility constraint}

Not every pair of search words has the single-crossing property.  The
Bellman recursion imposes an additional dynamic admissibility
constraint.  In the positive-determinant regime the movement update is
an increasing contraction in log odds.  After two trajectories are
forced to take opposite first actions, their log odds after movement
differ by less than the combined search shift $\Lambda$.  At every
later epoch they use the same monotone stage selector.  If their actions
differ, that selector assigns site~2 to the lower-belief trajectory and
site~1 to the higher-belief trajectory; the resulting search shifts
reverse the belief order and change the count discrepancy from one to
zero or from zero to one.  The discrepancy is consequently trapped
between these two values.

Combinatorially, a height-one pair is a lattice path on a two-level
ladder.  Binary words are standard objects in combinatorics on
words~\cite{lothaire1997}, while their encoding by lattice paths is
classical~\cite[Section~1.2]{stanley2011}.  Adjacent inversion swaps are
also fundamental in weak-order
constructions~\cite[Chapter~3]{bjornerbrenti2005}.  What is special here
is not the existence of local adjacent moves, but the stronger conclusion
of Lemma~\ref{lem:chain}:
for a two-level discrepancy path, the left prefixes of all required
adjacent moves can be chosen to be nested.  This is exactly the
condition needed by projective interval nesting.

There is a related line of work in stochastic control in which symbolic
itineraries are analysed using combinatorics on words.  Dance and
Silander~\cite{dance2015} used mechanical words in a Kalman-filter
restless-bandit problem to prove indexability, and their later
treatment~\cite{dance2019} relates threshold-policy itineraries to
Christoffel and other mechanical words in costly-observation control
problems.
The similarity is that a threshold policy converts a one-dimensional
controlled dynamical system into a binary word.  The role of the word
structure is different here, however.  We compare two finite action
sequences generated by different first actions along histories in which
every search is unsuccessful.  The height-one condition controls their
\emph{relative} prefix counts; this relative constraint is what makes
the common projective separator possible.

\subsection{Relation to structural POMDP theory}

The problem fits naturally into the structural theory of POMDPs.
Lovejoy~\cite{lovejoy1987} gives likelihood-ratio monotonicity results,
while Krishnamurthy~\cite[Chapters~10--12]{krishnamurthy2016} develops
the roles of total positivity, monotone likelihood-ratio order and
submodularity in obtaining monotone value functions and policies.  The present
proof uses the same broad order-theoretic language but at a different
level.  Rather than verifying one-step supermodularity of the Bellman
operator, it applies total positivity to matrix products associated with
entire sequences of unsuccessful searches.  It then uses the
height-one restriction on pairs of search words generated when both
trajectories follow the same monotone stage selectors.  In this sense
Theorem~\ref{thm:common-separator} is a multi-step single-crossing
certificate.

\subsection{Ties and the meaning of the conjecture}

A useful conceptual point is that the conjecture is a statement about
the existence of a threshold \emph{selector}.  It does not require the
action differential $D$ to be strictly decreasing, to have a unique
zero, or to cross transversally.  The proof establishes the ordering
property needed for such a selector: there is no pair $p<p'$ with
site~1 strictly better at $p$ and site~2 strictly better at $p'$.  Ties
can then be assigned to site~2 below a chosen threshold and to site~1
above it.  This weaker formulation is both robust under parameter limits
and sufficient for optimal control.

\subsection{Possible extensions}

Theorem~\ref{thm:common-separator} is intrinsically two-dimensional: the
projective image of the positive cone is an interval, and total
positivity gives a complete order of its boundary rays.  For three or
more target sites one should expect a higher-dimensional cone rather
than a scalar projective interval, while the word discrepancy becomes a
path in a higher-dimensional lattice.  It would be interesting to know
whether an analogue based on nested simplicial cones can yield useful
partial monotonicity results for multi-site moving-target search.
Another direction is to identify other controlled-sensing POMDPs in
which the two search words generated by forcing different current
actions and then applying the same stage selectors satisfy a
bounded-width lattice condition analogous to \eqref{eq:heightone}; the
same combination of adjacent moves and separators may then provide
threshold results outside the standard one-step supermodularity
framework.

\section*{Acknowledgements}
OpenAI's ChatGPT (GPT-5.6 Sol) was used to assist with
discovering and developing the proof, preparing an initial draft, and
refining the wording.  The author independently verified the
mathematical arguments and references and takes full responsibility for
the final content.

\bibliographystyle{emss}
\bibliography{references}

\appendix
\section{Detailed cobweb proof for the negative-determinant regime}
\label{app:negative}

This appendix proves Theorem~\ref{thm:negative} under
\eqref{eq:negative-strict} by reorganising the cobweb proof of MacPhee
and Jordan~\cite[Section~4]{macphee1995}.  Here a \emph{cobweb} is the
finite collection of belief intervals obtained by repeatedly taking
inverse images of a central
crossing point under the two failed-search maps; the intervals alternate
between the two sides of the threshold.

\subsection{Affine Bellman calculus}

For a function $f:[0,1]\to\R$, define the individual branch operators
\begin{equation}
(\mathcal B_i f)(p)
:=
C_i+q_i(p)f(F_i(p)),
\qquad i=1,2.
\label{eq:negative-Bi}
\end{equation}
Thus $\mathcal T f=\min\{\mathcal B_1f,\mathcal B_2f\}$.  Powers of
$\mathcal B_i$ denote composition.  When $f$ is affine, write
\begin{equation}
m(f):=f(1)-f(0)
\label{eq:negative-slope}
\end{equation}
for its slope.  Recall $\beta_i=1-\alpha_i$ and set
\begin{equation}
\vartheta_1:=\alpha_1a+b-1,
\qquad
\vartheta_2:=a+\alpha_2b-1.
\label{eq:negative-vartheta}
\end{equation}
Under \eqref{eq:negative-strict},
\begin{equation}
-1<\vartheta_1<0,
\qquad
-1<\vartheta_2<0.
\label{eq:negative-vartheta-range}
\end{equation}

\begin{lemma}[Affine slope identities]
\label{lem:negative-affine}
If $f$ is affine, then
\begin{align}
m(\mathcal B_1f)
&=
\vartheta_1m(f)-\beta_1f(0),
\label{eq:negative-B1-slope}\\
m(\mathcal B_2f)
&=
\vartheta_2m(f)+\beta_2f(1).
\label{eq:negative-B2-slope}
\end{align}
Consequently, if $f(p)>0$ for every $p\in[0,1]$,
\begin{align}
m(f)\ge0
&\Longrightarrow
m(\mathcal B_1f)<0,
\label{eq:negative-sign1}\\
m(f)\le0
&\Longrightarrow
m(\mathcal B_2f)>0.
\label{eq:negative-sign2}
\end{align}
Under the same assumption that $f$ is strictly positive on $[0,1]$,
the two branches with a common continuation $f$ satisfy
\begin{equation}
m(\mathcal B_1f-\mathcal B_2f)
=
-\beta_1f(a)-\beta_2f(1-b)<0.
\label{eq:negative-common-cont}
\end{equation}
\end{lemma}

\begin{proof}
Write $f(p)=f(0)+m(f)p$.  Since
\[
q_1(p)F_1(p)
=(1-b)+\bigl(\alpha_1a-(1-b)\bigr)p
\]
and
\[
q_2(p)F_2(p)
=\alpha_2(1-b)+\bigl(a-\alpha_2(1-b)\bigr)p,
\]
the coefficients of $p$ in the two branch costs are
\[
-\beta_1f(0)+\vartheta_1m(f)
\quad\text{and}\quad
\vartheta_2m(f)+\beta_2f(1),
\]
which proves \eqref{eq:negative-B1-slope}--\eqref{eq:negative-B2-slope}.
The sign implications follow from \eqref{eq:negative-vartheta-range}
and positivity of $f$.

Subtracting the two slope identities and using
$\vartheta_1-\vartheta_2=-\beta_1a+\beta_2b$ gives
\[
m(\mathcal B_1f-\mathcal B_2f)
=-\beta_1\bigl[f(0)+a m(f)\bigr]
 -\beta_2\bigl[f(0)+(1-b)m(f)\bigr],
\]
which is \eqref{eq:negative-common-cont}.
\end{proof}

Equation \eqref{eq:negative-common-cont} has a direct policy
interpretation.  Fix an affine cost-to-go function $f$ to be used after
the current search fails and movement occurs, regardless of which site
is searched now.  Then $\mathcal B_1f-\mathcal B_2f$ is the difference
between the two current-action costs under this common future rule, and
its slope is strictly negative.  The cobweb verification will also
compare $f$ with $\mathcal B_i^2f$, the cost of searching site~$i$ at
two successive epochs and then using continuation cost $f$.  The next
lemma gives the required slope comparisons.

\begin{lemma}[Two-step slope flattening]
\label{lem:negative-flattening}
Let $f$ be affine and strictly positive on $[0,1]$.
\begin{enumerate}[label=\textnormal{(\roman*)},leftmargin=2.2em]
\item If $m(f)>0$, then
\begin{equation}
m(\mathcal B_1^2f)<m(f).
\label{eq:negative-flatten1}
\end{equation}
\item If $m(f)<0$, then
\begin{equation}
m(\mathcal B_2^2f)>m(f).
\label{eq:negative-flatten2}
\end{equation}
\end{enumerate}
\end{lemma}

\begin{proof}
Applying \eqref{eq:negative-B1-slope} twice gives
\begin{equation}
m(\mathcal B_1^2f)
=
\vartheta_1^2m(f)
-\beta_1\!
\left[
C_1+(1-b)m(f)+(\alpha_1a+b)f(0)
\right].
\label{eq:negative-B11}
\end{equation}
If $m(f)>0$, the bracket is strictly positive and
$\vartheta_1^2<1$, proving \eqref{eq:negative-flatten1}.

Similarly,
\begin{equation}
m(\mathcal B_2^2f)
=
\vartheta_2^2m(f)
+\beta_2\!
\left[
C_2-(1-a)m(f)+(a+\alpha_2b)f(1)
\right].
\label{eq:negative-B22}
\end{equation}
If $m(f)<0$, the bracket is strictly positive and
$(\vartheta_2^2-1)m(f)>0$, proving \eqref{eq:negative-flatten2}.
\end{proof}

\subsection{Orientation reversal and finite cobwebs}

In log odds,
\[
\eta:=\operatorname{logit}(p):=\log\frac{p}{1-p},
\]
the movement step is the same map as in Section~\ref{sec:finite},
\[
\Phi(\eta)
=\log\frac{ae^\eta+1-b}{(1-a)e^\eta+b}.
\]

\begin{lemma}[Orientation-reversing contraction]
\label{lem:negative-contraction}
Under \eqref{eq:negative-strict},
\begin{equation}
-1<\Phi'(\eta)<0
\qquad(\eta\in\R).
\label{eq:negative-Phi}
\end{equation}
The maps $F_1,F_2$ are strictly decreasing contractions in log odds,
\begin{equation}
F_i([0,1])=[a,1-b],
\label{eq:negative-range}
\end{equation}
and
\begin{equation}
F_2(p)<F_1(p)
\qquad(0<p<1).
\label{eq:negative-F-order}
\end{equation}
Each $F_i$ has a unique fixed point $P_i\in(0,1)$, and
\begin{equation}
P_2<P_1.
\label{eq:negative-P-order}
\end{equation}
\end{lemma}

\begin{proof}
With $x=e^\eta$,
\[
\Phi'(\eta)
=\frac{\Delta x}
{(ax+1-b)((1-a)x+b)}.
\]
The derivative is negative.  Moreover,
\begin{align*}
&(ax+1-b)((1-a)x+b)-(1-a-b)x\\
&\qquad
=a(1-a)x^2+2abx+b(1-b)>0,
\end{align*}
so $|\Phi'(\eta)|<1$.  Since this derivative is continuous and tends
to zero as $\eta\to\pm\infty$, its absolute value has a global maximum
$\kappa_\Phi<1$.  To apply this bound to $F_1$ and $F_2$, take
$p,p'\in(0,1)$ and write
$\eta=\operatorname{logit}(p)$ and
$\eta'=\operatorname{logit}(p')$.  With
$\ell_i=-\log\alpha_i$ as in \eqref{eq:logodds}, the search shifts
\eqref{eq:search-shift} give
\[
 \operatorname{logit}(F_1(p))=\Phi(\eta-\ell_1),
 \qquad
 \operatorname{logit}(F_2(p))=\Phi(\eta+\ell_2).
\]
Both right-hand sides are strictly decreasing functions of $\eta$.
Moreover, the mean value theorem and the definition of $\kappa_\Phi$ give,
for $i=1,2$,
\[
 \left|
 \operatorname{logit}(F_i(p))
 -\operatorname{logit}(F_i(p'))
 \right|
 \le \kappa_\Phi|\eta-\eta'|.
\]
Thus $F_1$ and $F_2$ are strictly decreasing contractions when beliefs
are measured in log odds.

To compare the two maps, fix $p\in(0,1)$.  Conditional on an unsuccessful
search of site~1, the posterior probability that the target is at
site~1, immediately after the search and before movement, is
\[
 \frac{\alpha_1p}{\alpha_1p+1-p}.
\]
Conditional on an unsuccessful search of site~2, the corresponding
posterior probability is
\[
 \frac{p}{p+\alpha_2(1-p)}.
\]
Because $0<\alpha_1,\alpha_2<1$, these two posterior probabilities
satisfy
\[
 \frac{\alpha_1p}{\alpha_1p+1-p}
 <p<
 \frac{p}{p+\alpha_2(1-p)}.
\]
The movement step sends a pre-movement probability $z$ of being at
site~1 to
\[
 az+(1-b)(1-z).
\]
This affine function has slope $\Delta<0$ and therefore reverses the
preceding inequalities.  Its values at the left and right posterior
probabilities are $F_1(p)$ and $F_2(p)$, respectively, so
$F_2(p)<F_1(p)$.  Direct evaluation gives
$F_i(0)=1-b$ and $F_i(1)=a$, proving \eqref{eq:negative-range}.
Since $F_i$ is continuous, $F_i(0)>0$, and $F_i(1)<1$, it has a fixed
point in $(0,1)$.  Because $F_i$ is strictly decreasing,
$p\mapsto F_i(p)-p$ is strictly decreasing, so this fixed point is
unique.  To compare the two fixed points, observe that
\[
 F_2(P_1)<F_1(P_1)=P_1.
\]
Since $p\mapsto F_2(p)-p$ is strictly decreasing and vanishes only at
$P_2$, this inequality implies $P_2<P_1$.
\end{proof}

Let $U_i:=F_i^{-1}:[a,1-b]\to[0,1]$.

\begin{lemma}[Self-cobweb inequalities]
\label{lem:negative-self-cobweb}
For each $i\in\{1,2\}$ and $p\in[a,1-b]$,
\begin{align}
p<P_i
&\Longrightarrow
p<P_i<F_i(p)<U_i(p),
\label{eq:negative-self-low}\\
p>P_i
&\Longrightarrow
U_i(p)<F_i(p)<P_i<p.
\label{eq:negative-self-high}
\end{align}
\end{lemma}

\begin{proof}
Let $z$ and $z_i$ be the log odds of $p$ and $P_i$, respectively.  If
$p<P_i$, then $z<z_i$.
Because the log-odds form of $F_i$ is decreasing and fixes $z_i$,
$\operatorname{logit}(F_i(p))>z_i$ and
$\operatorname{logit}(U_i(p))>z_i$.  Applying the contraction estimate
first to $p$ and $P_i$, and then to $U_i(p)$ and $P_i$, gives
\[
 0<\operatorname{logit}(F_i(p))-z_i
 <z_i-z
 <\operatorname{logit}(U_i(p))-z_i.
\]
Since the logistic map is increasing, this proves
\eqref{eq:negative-self-low}.  The proof of
\eqref{eq:negative-self-high} is identical with all inequalities
reversed.
\end{proof}

\begin{lemma}[Finite inverse cobweb]
\label{lem:negative-finite-cobweb}
Suppose a sequence $(x_k)_{k\ge0}$ satisfies either
\[
 x_{k+1}=U_2U_1(x_k)
 \qquad\text{or}\qquad
 x_{k+1}=U_1U_2(x_k)
\]
for every $k\ge0$ for which the relevant composition is defined.  More
precisely, $U_2U_1(x_k)$ is defined when
$x_k\in[a,1-b]$ and $U_1(x_k)\in[a,1-b]$, while $U_1U_2(x_k)$ is
defined when $x_k\in[a,1-b]$ and $U_2(x_k)\in[a,1-b]$.  If
$x_1\ne x_0$, then the recursion terminates after finitely many steps.
\end{lemma}

\begin{proof}
Let $\kappa_\Phi<1$ be the global log-odds contraction constant from
Lemma~\ref{lem:negative-contraction}.  Every two-step forward
composition has Lipschitz constant at most $\kappa_\Phi^2$.  Consider, for
example, the first recursion,
\[
x_{k+1}=U_2U_1(x_k),
\]
so that $x_k=(F_1\circ F_2)(x_{k+1})$.  Writing
$z_k=\log\!\bigl(x_k/(1-x_k)\bigr)$, we obtain
\[
|z_{k+1}-z_k|
\ge
\kappa_\Phi^{-2}|z_k-z_{k-1}|
\]
whenever the three points are defined.  The first gap is nonzero, so the
iteration of this inequality gives
\[
 |z_{k+1}-z_k|
 \ge \kappa_\Phi^{-2k}|z_1-z_0|.
\]
The right-hand side tends to infinity.  If the recursion did not
terminate, then $x_k\in[a,1-b]$ for every $k$, and hence
\[
 |z_{k+1}-z_k|
 \le
 \operatorname{logit}(1-b)-\operatorname{logit}(a)<\infty
 \qquad(k\ge0).
\]
This contradicts the preceding lower bound.  The argument for the
recursion governed by $U_1U_2$ is identical.
\end{proof}

\subsection{Transporting a crossing through a common history}

For a chronological word $u=u_1\cdots u_k$, write
\[
\mathcal B_u
:=
\mathcal B_{u_1}\circ\cdots\circ\mathcal B_{u_k},
\qquad
F_u
:=
F_{u_k}\circ\cdots\circ F_{u_1}.
\]
Thus $\mathcal B_{u_1}$ represents the first search, although it is the
outermost branch operator acting on the continuation function.  We also
write
\[
q_u(p):=\rho(p)A_u\1
\]
for the probability of surviving every search in the prefix $u$.
For the empty word, set
\[
 \mathcal B_\varnothing f:=f,
 \qquad
 F_\varnothing(p):=p,
 \qquad
 q_\varnothing(p):=1.
\]

\begin{lemma}[Common-prefix crossing identity]
\label{lem:negative-pullback}
For every finite word $u$ and functions $f,g:[0,1]\to\R$,
\begin{equation}
(\mathcal B_uf)(p)-(\mathcal B_ug)(p)
=
q_u(p)\bigl[f(F_u(p))-g(F_u(p))\bigr],
\label{eq:negative-pullback}
\end{equation}
where $q_u(p)>0$ under \eqref{eq:negative-strict}.  Consequently, if $f-g$
vanishes at $z$, then the two common-prefix costs
$\mathcal B_uf$ and $\mathcal B_ug$ are equal at every initial belief
$p$ satisfying $F_u(p)=z$.

Moreover, suppose $z\in(0,1)$, $f(z)=g(z)$, and, for some
$0<\epsilon<\min\{z,1-z\}$, either
\[
 f(x)<g(x)\quad(z-\epsilon<x<z),
 \qquad
 f(x)>g(x)\quad(z<x<z+\epsilon),
\]
or both inequalities are reversed.  If $F_u(p)=z$, then
$\mathcal B_uf-\mathcal B_ug$ has the same left--right sign pattern at
$p$ when $|u|$ is even and the reversed left--right sign pattern when
$|u|$ is odd.
\end{lemma}

\begin{proof}
For one action the immediate search costs cancel, leaving
\[
q_i(p)\bigl[f(F_i(p))-g(F_i(p))\bigr].
\]
Iteration proves \eqref{eq:negative-pullback}.  Since each $F_i$ is
strictly decreasing, $F_u$ is strictly increasing when $|u|$ is even
and strictly decreasing when $|u|$ is odd.  Thus, near a point $p$ with
$F_u(p)=z$, the two sides of $p$ are mapped to the same respective
sides of $z$ in the even case and to the opposite sides in the odd
case.  The factor $q_u(p)$ in \eqref{eq:negative-pullback} is positive
and therefore does not change either sign.  This proves the stated
left--right sign rule.
\end{proof}

We shall describe the point $p=F_u^{-1}(z)$, when it exists, as the
crossing at $z$ \emph{pulled back} through the common history $u$.

\subsection{Three central comparisons}

Let $W_1$ and $W_2$ denote the infinite-horizon expected costs of always
searching sites~1 and~2, respectively.  These costs are finite under
\eqref{eq:negative-strict}: after the first movement, always searching
site~1 has detection probability at least $a(1-\alpha_1)>0$ at every
search, and always searching site~2 has detection probability at least
$b(1-\alpha_2)>0$.  The truncated cost functions of both fixed policies
are affine in the initial belief and converge uniformly because the
corresponding survival probabilities decay geometrically.  Their limits
$W_1,W_2$
are therefore positive affine functions and satisfy
\begin{equation}
W_i=\mathcal B_iW_i,\qquad i=1,2.
\label{eq:negative-W-fixed}
\end{equation}

\begin{lemma}[Always-search slopes]
\label{lem:negative-W-slopes}
One has
\begin{equation}
m(W_1)=-\frac{C_1}{1-b}<0,
\qquad
m(W_2)=\frac{C_2}{1-a}>0.
\label{eq:negative-W-slopes}
\end{equation}
\end{lemma}

\begin{proof}
At $p=0$ one has $q_1(0)=1$ and
$F_1(0)=1-b$, so $W_1=\mathcal B_1W_1$ gives
\[
W_1(0)=C_1+W_1(1-b)
=C_1+W_1(0)+(1-b)m(W_1),
\]
and hence $m(W_1)=-C_1/(1-b)$.  For $W_2$, evaluate
$W_2=\mathcal B_2W_2$ at $p=1$, where $q_2(1)=1$ and $F_2(1)=a$, to
obtain
\[
 W_2(1)=C_2+W_2(a)
 =C_2+W_2(1)-(1-a)m(W_2).
\]
Therefore $m(W_2)=C_2/(1-a)$.
\end{proof}

Define the three affine comparison functions
\begin{equation}
G_0:=W_1-W_2,
\qquad
G_1:=W_1-\mathcal B_2W_1,
\qquad
G_2:=\mathcal B_1W_2-W_2.
\label{eq:negative-G}
\end{equation}
By Lemma~\ref{lem:negative-W-slopes}, $G_0$ is strictly decreasing, and
by \eqref{eq:negative-common-cont}, so are $G_1$ and $G_2$.

\begin{lemma}[Fixed-point matching]
\label{lem:negative-fixed-match}
At the two fixed points,
\begin{align}
G_1(P_2)
&=
[1-q_2(P_2)]G_0(P_2),
\label{eq:negative-G1-match}\\
G_2(P_1)
&=
[1-q_1(P_1)]G_0(P_1).
\label{eq:negative-G2-match}
\end{align}
\end{lemma}

\begin{proof}
Since $F_2(P_2)=P_2$,
\[
W_2(P_2)=C_2+q_2(P_2)W_2(P_2),
\]
whereas
\[
(\mathcal B_2W_1)(P_2)
=C_2+q_2(P_2)W_1(P_2).
\]
Subtracting gives \eqref{eq:negative-G1-match}.  The second identity is
obtained in the same way from
\[
 W_1(P_1)=C_1+q_1(P_1)W_1(P_1)
\]
and
\[
 (\mathcal B_1W_2)(P_1)=C_1+q_1(P_1)W_2(P_1).
\]
\end{proof}

Under \eqref{eq:negative-strict}, both $1-q_2(P_2)$ and
$1-q_1(P_1)$ are positive.  Hence
\eqref{eq:negative-G1-match} shows that $G_1(P_2)$ has the same sign as
$G_0(P_2)$, while \eqref{eq:negative-G2-match} shows that $G_2(P_1)$
has the same sign as $G_0(P_1)$.  Together with the strict decrease of
$G_0,G_1$, and $G_2$, these sign relations yield the three cases in the
next proposition.

\begin{proposition}[Three negative-determinant configurations]
\label{prop:negative-three}
If neither $G_0(P_2)$ nor $G_0(P_1)$ is zero, exactly one of the
following occurs.
\begin{enumerate}[label=\textnormal{(\roman*)},leftmargin=2.2em]
\item If $G_0(P_2)<0$, then either $G_1\le0$ on $[0,1]$, in which case
site~1 minimises the Bellman right-hand side at every belief, or $G_1$
has a unique zero
$p_{\mathrm c}$ satisfying $0<p_{\mathrm c}<P_2$.
\item If
\[
G_0(P_2)>0>G_0(P_1),
\]
then $G_0$ has a unique zero $p_{\mathrm c}$ satisfying
$P_2<p_{\mathrm c}<P_1$.
\item If $G_0(P_1)>0$, then either $G_2\ge0$ on $[0,1]$, in which case
site~2 minimises the Bellman right-hand side at every belief, or $G_2$
has a unique zero
$p_{\mathrm c}$ satisfying $P_1<p_{\mathrm c}<1$.
\end{enumerate}
\end{proposition}

\begin{proof}
If $G_0(P_2)<0$, then $G_1(P_2)<0$ by
\eqref{eq:negative-G1-match}.  Since $G_1$ is decreasing, either
$G_1(0)\le0$ and $G_1\le0$ everywhere, or it has one zero in
$(0,P_2)$.  In the first case
\[
W_1=\mathcal B_1W_1\le\mathcal B_2W_1,
\]
so $W_1$ is a Bellman fixed point with site~1 minimising everywhere.
The middle assertion follows from strict monotonicity of $G_0$ and
$P_2<P_1$.  In the third case,
\eqref{eq:negative-G2-match} gives $G_2(P_1)>0$.  Since $G_2$ is
decreasing, either $G_2(1)\ge0$ and hence $G_2\ge0$ on $[0,1]$, or
$G_2$ has a unique zero in $(P_1,1)$.
\end{proof}

The two one-sided configurations are mirror images.  We verify the
lower one-sided case and the interlaced middle case.

\subsection{The one-sided cobweb}
\label{subsec:negative-one-sided}

Consider the second alternative in case~(i), and let $p_{\mathrm c}$
denote the unique zero of $G_1$.  Then
\begin{equation}
p_{\mathrm c}<P_2<P_1,
\qquad
W_1(p_{\mathrm c})=(\mathcal B_2W_1)(p_{\mathrm c}).
\label{eq:negative-class1-anchor}
\end{equation}
Put
\[
\varphi_0^+:=W_1,
\qquad
x_0:=p_{\mathrm c}.
\]
Starting from $x_0$, define $y_k=U_1(x_k)$ when
$x_k\in[a,1-b]$, and then define $x_{k+1}=U_2(y_k)$ when, in addition,
$y_k\in[a,1-b]$.  Thus the recursion stops as soon as either domain
condition fails.  In formulas,
\begin{equation}
y_k:=U_1(x_k),
\qquad
x_{k+1}:=U_2(y_k).
\label{eq:negative-class1-points}
\end{equation}

\begin{lemma}[One-sided cobweb order]
\label{lem:negative-class1-order}
The points generated by \eqref{eq:negative-class1-points} satisfy
\begin{equation}
\cdots<x_2<x_1<x_0=p_{\mathrm c}<y_0<y_1<y_2<\cdots
\label{eq:negative-class1-order}
\end{equation}
until the inverse recursion terminates.
\end{lemma}

\begin{proof}
We prove $x_{k+1}<x_k$ inductively.  The induction starts from
$x_0=p_{\mathrm c}<P_1$.  Suppose $x_k$ has been constructed and $x_k<P_1$.
Whenever $y_k$ and $x_{k+1}$ are defined,
Lemma~\ref{lem:negative-self-cobweb} together with
\eqref{eq:negative-F-order} gives
\[
U_1(x_k)>F_1(x_k)>F_2(x_k).
\]
Applying the decreasing inverse $U_2$ yields
\[
x_{k+1}=U_2(U_1(x_k))
<U_2(F_2(x_k))
=x_k.
\]
In particular, $x_{k+1}<P_1$, which closes the induction.
Since $U_1$ is decreasing, $x_{k+1}<x_k$ implies
$y_{k+1}>y_k$ whenever both points exist.  Finally,
$p_{\mathrm c}<P_1<F_1(p_{\mathrm c})<U_1(p_{\mathrm c})=y_0$ whenever $y_0$ exists.
\end{proof}

Define affine continuation pieces recursively by
\begin{equation}
\varphi_{k+1}^-:=\mathcal B_2\varphi_k^+,
\qquad
\varphi_{k+1}^+:=\mathcal B_1\varphi_{k+1}^-,
\qquad k\ge0.
\label{eq:negative-class1-lines}
\end{equation}
We now define a function $\bar V$ on $[0,1]$ and will prove below that
it satisfies the Bellman equation.  Lemma~\ref{lem:negative-verification}
will then identify $\bar V$ with $V$.  A \emph{cobweb cell} is an interval
on which $\bar V$ is represented by one of these affine functions.  All
displayed cells are included only when their
endpoints exist.  If $y_0=U_1(p_{\mathrm c})$ is undefined, set
$\bar V=\varphi_0^+$ on $[p_{\mathrm c},1]$ and
$\bar V=\varphi_1^-$ on $[0,p_{\mathrm c}]$.  Otherwise set
$\bar V=\varphi_0^+$ on $[p_{\mathrm c},y_0]$ and, whenever the relevant
endpoints exist, set
\begin{equation}
\varphi_{k+1}^-\quad\text{on }[x_{k+1},x_k],
\qquad
\varphi_{k+1}^+\quad\text{on }[y_k,y_{k+1}].
\label{eq:negative-class1-cells}
\end{equation}
Once $y_0$ exists, there are two possible ways for this recursion to
terminate.  If $y_k$
has been constructed but $x_{k+1}=U_2(y_k)$ is undefined, set
\[
 \bar V=\varphi_{k+1}^-\quad\text{on }[0,x_k],
 \qquad
 \bar V=\varphi_{k+1}^+\quad\text{on }[y_k,1].
\]
If $x_{k+1}$ has been constructed but
$y_{k+1}=U_1(x_{k+1})$ is undefined, set
\[
 \bar V=\varphi_{k+2}^-\quad\text{on }[0,x_{k+1}],
 \qquad
 \bar V=\varphi_{k+1}^+\quad\text{on }[y_k,1].
\]
Together with the previously displayed cells, these assignments define
$\bar V$ on all of $[0,1]$.  At a shared endpoint either adjacent
formula may be used provisionally; the crossing argument below shows
that the formulas agree there.

Consider the rule that searches site~2 below $p_{\mathrm c}$ and site~1
above $p_{\mathrm c}$.  For every displayed nonterminal cell, the next
belief under this rule lies in the interval on which the continuation
function in \eqref{eq:negative-class1-lines} is used.  For
$k\ge1$,
\begin{equation}
F_2([x_{k+1},x_k])=[y_{k-1},y_k],
\label{eq:negative-class1-intended-low}
\end{equation}
while
\begin{equation}
F_2([x_1,p_{\mathrm c}])\subset[p_{\mathrm c},y_0]
\label{eq:negative-class1-intended-first}
\end{equation}
for the first lower cell, and
\begin{equation}
F_1([y_k,y_{k+1}])=[x_{k+1},x_k].
\label{eq:negative-class1-intended-high}
\end{equation}
On a terminal interval the corresponding equality above becomes an
inclusion into the last continuation interval.  This follows directly
from the terminal assignments: an undefined $U_i(z)$ means
$z\notin[a,1-b]=F_i([0,1])$, so $F_i(p)$ cannot cross $z$.
Thus this rule satisfies
$\bar V=\mathcal B_2\bar V$ below $p_{\mathrm c}$ and
$\bar V=\mathcal B_1\bar V$ above $p_{\mathrm c}$.  To prove the Bellman
equation, it remains only to show that the competing action has no
smaller cost.

The anchor is the equality
$\varphi_0^+(x_0)=\varphi_1^-(x_0)$.  Pulling it back alternately through
$F_1$ and $F_2$ by Lemma~\ref{lem:negative-pullback} gives
\[
 \varphi_k^+(y_k)=\varphi_{k+1}^+(y_k)\quad(k\ge0)
\]
and
\[
 \varphi_k^-(x_k)=\varphi_{k+1}^-(x_k)\quad(k\ge1)
\]
whenever the indicated points and affine formulas exist.  Thus the two
adjacent formulas agree at each $x_k$ and $y_k$; these points are the
switching points of the piecewise-affine function $\bar V$.  In
particular, $\bar V$ is continuous.  Each affine function
is strictly positive because $W_1>0$ and each branch operation adds a
positive immediate cost.  Moreover, \eqref{eq:negative-sign2} gives a
positive slope for $\varphi_{k+1}^-$ whenever $\varphi_k^+$ has negative
slope, and \eqref{eq:negative-sign1} then gives a negative slope for
$\varphi_{k+1}^+$.  Lemma~\ref{lem:negative-W-slopes} supplies the
initial signs.
It remains to compare slopes within each family.  The affine function
\[
 G_1=\varphi_0^+-\varphi_1^-
\]
is strictly decreasing and vanishes at $x_0=p_{\mathrm c}$.  Hence it
is positive to the left of $x_0$ and negative to the right.  When
$y_0$ exists, pulling this crossing back through $F_1$ gives, at
$y_0=U_1(x_0)$, the
difference
\[
 \varphi_0^+-\varphi_1^+
 =\mathcal B_1\varphi_0^+-\mathcal B_1\varphi_1^-.
\]
By Lemma~\ref{lem:negative-pullback}, this difference is negative to
the left of $y_0$ and positive to the right.  Because it is affine, its
slope is positive, and therefore
$m(\varphi_0^+)>m(\varphi_1^+)$.  If $x_1$ exists, pulling this new
crossing back through $F_2$ gives, at $x_1=U_2(y_0)$, the difference
\[
 \varphi_1^--\varphi_2^-
 =\mathcal B_2\varphi_0^+-\mathcal B_2\varphi_1^+.
\]
Its sign is again positive on the left and negative on the right, so
its slope is negative and
$m(\varphi_1^-)<m(\varphi_2^-)$.  Repeating these two pullbacks
alternately, for as long as the inverse points exist, proves
\begin{align}
0
&<
m(\varphi_1^-)<m(\varphi_2^-)<m(\varphi_3^-)<\cdots,
\label{eq:negative-class1-L-order}\\
0
&>
m(\varphi_0^+)>m(\varphi_1^+)>m(\varphi_2^+)>\cdots.
\label{eq:negative-class1-H-order}
\end{align}
Reading the cells from left to right, these inequalities say that their
slopes are nonincreasing.  Thus $\bar V$ is concave as well as
continuous.

On each cobweb cell, call the affine function defining $\bar V$ there
the \emph{active piece}.  The \emph{competing branch} is the cost of
taking the action not prescribed by the threshold rule for the current
search and then using continuation $\bar V$.  Thus, on a lower cell,
the active action is site~2 and the competing branch is
$\mathcal B_1\bar V$.  Because $\bar V$ is piecewise affine, the
competing branch may itself have several affine subpieces on one cell.
Whenever $y_j$ exists, the endpoint inequalities
\[
F_2(x_j)<F_1(x_j)<U_1(x_j)=y_j
\]
hold.  Consequently, for $k\ge1$, if $y_{k+1}$ exists, then
\[
F_1([x_{k+1},x_k])
\subset[y_{k-1},y_{k+1}].
\]
For the first lower cell, when $y_1$ exists, the corresponding inclusion
is
\[
F_1([x_1,x_0])\subset[p_{\mathrm c},y_1].
\]
If $y_{k+1}$ is undefined, the terminal assignment instead makes
$\varphi_{k+1}^+$ the continuation function on $[y_k,1]$; thus no
additional continuation piece occurs.  Hence, in every case, the
competing site~1 branch has only the affine subpieces
\[
\mathcal B_1\varphi_k^+,
\qquad
\mathcal B_1\varphi_{k+1}^+.
\]
For $k\ge1$ these are
$\mathcal B_1^2\varphi_k^-$ and
$\mathcal B_1^2\varphi_{k+1}^-$; for $k=0$ the first subpiece is
$\mathcal B_1\varphi_0^+=\varphi_0^+$, while the second is
$\mathcal B_1^2\varphi_1^-$.  Lemma~\ref{lem:negative-flattening} and
\eqref{eq:negative-class1-L-order} therefore imply that every affine
subpiece $g$ of $\mathcal B_1\bar V$ on this lower cell satisfies
\begin{equation}
m(g)<m(\varphi_{k+1}^-).
\label{eq:negative-class1-wrong-low-slope}
\end{equation}

Similarly, when $x_{k+2}$ exists, the endpoints of an upper
$\varphi_{k+1}^+$-cell satisfy
\[
x_{k+1}=U_2(y_k)<F_2(y_k)<F_1(y_k)=x_k
\]
and
\[
x_{k+2}=U_2(y_{k+1})<F_2(y_{k+1})<F_1(y_{k+1})=x_{k+1}.
\]
Thus
\[
F_2([y_k,y_{k+1}])\subset[x_{k+2},x_k].
\]
If $x_{k+2}$ is undefined, the terminal assignment makes
$\varphi_{k+2}^-$ the continuation function on $[0,x_{k+1}]$.
Therefore, in either case, the competing site~2 branch meets only
$\varphi_{k+1}^-$ and $\varphi_{k+2}^-$.  Its affine subpieces are
$\mathcal B_2^2\varphi_k^+$ and
$\mathcal B_2^2\varphi_{k+1}^+$.  By
\eqref{eq:negative-class1-H-order}, the underlying slopes are at least
$m(\varphi_{k+1}^+)$, and Lemma~\ref{lem:negative-flattening} strictly
increases each of them.  Hence every affine subpiece $g$ of
$\mathcal B_2\bar V$ on this upper cell satisfies
\begin{equation}
m(g)>m(\varphi_{k+1}^+).
\label{eq:negative-class1-wrong-high-slope}
\end{equation}

The central $\varphi_0^+=W_1$ cell is not among the upper cells in
\eqref{eq:negative-class1-cells}.  If $y_0$ and $x_1$ exist, then
\[
 F_2([p_{\mathrm c},y_0])\subset[x_1,y_0],
\]
whereas if either inverse point is undefined, the corresponding
terminal assignment and the range $F_2([0,1])=[a,1-b]$ show directly
that no other continuation piece is encountered.  Thus on the central
cell the competing site~2 branch can
use only the continuation pieces $\varphi_0^+$ and
$\varphi_1^-$.  Its affine subpieces are therefore
$\mathcal B_2\varphi_0^+=\varphi_1^-$ and
$\mathcal B_2\varphi_1^-=\mathcal B_2^2\varphi_0^+$.  The first has
positive slope, while Lemma~\ref{lem:negative-flattening} gives
\[
 m(\mathcal B_2^2\varphi_0^+)>m(\varphi_0^+).
\]
Thus every affine subpiece of the competing branch on the central cell
has slope greater than $m(\varphi_0^+)$.

Define the excess costs of the competing actions by
\[
\mathcal E_1:=\mathcal B_1\bar V-\bar V\quad(p\le p_{\mathrm c}),
\qquad
\mathcal E_2:=\mathcal B_2\bar V-\bar V\quad(p\ge p_{\mathrm c}).
\]
They are continuous and piecewise affine.  The preceding slope
comparisons give $\mathcal E_1'<0$ on every lower affine subcell and
$\mathcal E_2'>0$ on every upper affine subcell, including the central
cell.  At $p_{\mathrm c}$,
the self-cobweb inequalities place
both $F_1(p_{\mathrm c})$ and $F_2(p_{\mathrm c})$ in the central
$\varphi_0^+=W_1$ cell.  Thus
\[
\mathcal B_1\bar V(p_{\mathrm c})=W_1(p_{\mathrm c})
\]
and, by \eqref{eq:negative-class1-anchor},
\[
\mathcal B_2\bar V(p_{\mathrm c})=W_1(p_{\mathrm c}).
\]
Hence $\mathcal E_1(p_{\mathrm c})=\mathcal E_2(p_{\mathrm c})=0$, and therefore
\begin{equation}
\mathcal E_1(p)\ge0\quad(p<p_{\mathrm c}),
\qquad
\mathcal E_2(p)\ge0\quad(p>p_{\mathrm c}).
\label{eq:negative-class1-excess}
\end{equation}
Thus $\bar V=\mathcal T\bar V$.  Site~2 minimises below
$p_{\mathrm c}$ and site~1 above it, so $p_{\mathrm c}$ is the
threshold of this selector.  The strict ordering in
Lemma~\ref{lem:negative-class1-order}, together with
Lemma~\ref{lem:negative-finite-cobweb}, shows that there are only
finitely many cells.  Every affine piece is obtained
from the positive function $W_1$ by finitely many branch operations, so
$\bar V$ is positive and bounded on $[0,1]$.  The upper one-sided case
is obtained by exchanging $(a,\alpha_1,C_1)$ with
$(b,\alpha_2,C_2)$ and replacing $p$ by $1-p$.

\subsection{The interlaced cobweb}
\label{subsec:negative-interlaced}

In case~(ii), let $p_{\mathrm c}$ denote the unique zero of $G_0$.
Then
\begin{equation}
P_2<p_{\mathrm c}<P_1,
\qquad
W_1(p_{\mathrm c})=W_2(p_{\mathrm c}).
\label{eq:negative-class3-anchor}
\end{equation}
Since $P_2,P_1\in[a,1-b]$ and $P_2<p_{\mathrm c}<P_1$, the point
$p_{\mathrm c}$ also belongs to $[a,1-b]$.  Thus the first inverse
points $\omega_0=U_2(p_{\mathrm c})$ and
$\zeta_0=U_1(p_{\mathrm c})$ exist.
Start with
\[
\lambda_0=\upsilon_0=p_{\mathrm c}
\]
and construct two inverse recursions.  In the first recursion, for
$k\ge0$, define
$\zeta_k=U_1(\lambda_k)$ when $\lambda_k\in[a,1-b]$, and then define
$\lambda_{k+1}=U_2(\zeta_k)$ when, in addition,
$\zeta_k\in[a,1-b]$.  In the second recursion, for $k\ge0$, define
$\omega_k=U_2(\upsilon_k)$ when $\upsilon_k\in[a,1-b]$, and then define
$\upsilon_{k+1}=U_1(\omega_k)$ when, in addition,
$\omega_k\in[a,1-b]$.  Each recursion stops as soon as one of its domain
conditions fails.  In formulas,
\begin{align}
\zeta_k&:=U_1(\lambda_k),
&
\lambda_{k+1}&:=U_2(\zeta_k),
\label{eq:negative-ab}\\
\omega_k&:=U_2(\upsilon_k),
&
\upsilon_{k+1}&:=U_1(\omega_k).
\label{eq:negative-cd}
\end{align}

\begin{lemma}[Interlacing]
\label{lem:negative-interlace}
The points generated by \eqref{eq:negative-ab}--\eqref{eq:negative-cd}
satisfy
\begin{equation}
\cdots<\lambda_2<\omega_1<\lambda_1<\omega_0<p_{\mathrm c}
<\zeta_0<\upsilon_1<\zeta_1<\upsilon_2<\zeta_2<\cdots
\label{eq:negative-interlace}
\end{equation}
whenever the displayed points exist.
\end{lemma}

\begin{proof}
Because $P_2<p_{\mathrm c}<P_1$,
\[
\omega_0=U_2(p_{\mathrm c})<p_{\mathrm c}<U_1(p_{\mathrm c})=\zeta_0.
\]
Since $U_2$ is decreasing,
\[
\lambda_1=U_2(\zeta_0)<U_2(p_{\mathrm c})=\omega_0,
\]
and since $U_1$ is decreasing,
\[
\upsilon_1=U_1(\omega_0)>U_1(p_{\mathrm c})=\zeta_0.
\]
Moreover $\lambda_1<\omega_0$ implies
\[
\zeta_1=U_1(\lambda_1)>U_1(\omega_0)=\upsilon_1.
\]
The inequality $\upsilon_1>\zeta_0$ also gives
\[
 \omega_1=U_2(\upsilon_1)<U_2(\zeta_0)=\lambda_1.
\]
Because $U_1$ and $U_2$ are decreasing, the compositions $U_2U_1$ and
$U_1U_2$ are increasing.  On the lower side,
\[
\lambda_{k+1}=(U_2U_1)(\lambda_k),
\qquad
\omega_{k+1}=(U_2U_1)(\omega_k).
\]
If $\lambda_{k+1}<\omega_k$, applying $U_2U_1$ gives
$\lambda_{k+2}<\omega_{k+1}$; if
$\omega_{k+1}<\lambda_{k+1}$, applying the same map gives
$\omega_{k+2}<\lambda_{k+2}$.  Hence
\[
 \lambda_{k+1}<\omega_k,
 \qquad
 \omega_{k+1}<\lambda_{k+1}
\]
for every $k$ for which these points exist.  On the upper side,
\[
\zeta_{k+1}=(U_1U_2)(\zeta_k),
\qquad
\upsilon_{k+1}=(U_1U_2)(\upsilon_k).
\]
Starting from $\zeta_0<\upsilon_1<\zeta_1$, applying the increasing map
$U_1U_2$ to all three terms proves
$\zeta_k<\upsilon_{k+1}<\zeta_{k+1}$ inductively.  The lower and upper
inequalities together give
\eqref{eq:negative-interlace}.
\end{proof}

Define four alternating families of affine continuation pieces.  The
superscript $-$ denotes a lower-belief piece and $+$ an upper-belief
piece.  Put
\[
\psi_{-1}^-:=W_2,
\qquad
\psi_{-1}^+:=W_1,
\]
and define
\begin{equation}
\phi_k^-:=\mathcal B_2\psi_{k-1}^+,
\qquad
\phi_k^+:=\mathcal B_1\psi_{k-1}^-,
\qquad
\psi_k^-:=\mathcal B_2\phi_k^+,
\qquad
\psi_k^+:=\mathcal B_1\phi_k^-,
\qquad k\ge0.
\label{eq:negative-class3-recursion}
\end{equation}
We now define a function $\bar V$ on $[0,1]$ and will prove below that
it satisfies the Bellman equation.  Lemma~\ref{lem:negative-verification}
will then identify $\bar V$ with $V$.  On the cobweb cells, set
\begin{align}
\psi_{-1}^-&\quad\text{on }[\omega_0,p_{\mathrm c}],
&
\psi_{-1}^+&\quad\text{on }[p_{\mathrm c},\zeta_0],
\label{eq:negative-class3-centre}\\
\phi_k^-&\quad\text{on }[\lambda_{k+1},\omega_k],
&
\psi_k^-&\quad\text{on }[\omega_{k+1},\lambda_{k+1}],
\label{eq:negative-class3-lower}\\
\phi_k^+&\quad\text{on }[\zeta_k,\upsilon_{k+1}],
&
\psi_k^+&\quad\text{on }[\upsilon_{k+1},\zeta_{k+1}].
\label{eq:negative-class3-upper}
\end{align}
Each displayed cell is used only when both endpoints exist.  The four
possible terminal assignments are as follows:
\begin{align*}
 \lambda_{k+1}\text{ undefined}
 &\quad\Longrightarrow\quad
 \bar V=\phi_k^-\text{ on }[0,\omega_k],\\
 \omega_{k+1}\text{ undefined}
 &\quad\Longrightarrow\quad
 \bar V=\psi_k^-\text{ on }[0,\lambda_{k+1}],\\
 \upsilon_{k+1}\text{ undefined}
 &\quad\Longrightarrow\quad
 \bar V=\phi_k^+\text{ on }[\zeta_k,1],\\
 \zeta_{k+1}\text{ undefined}
 &\quad\Longrightarrow\quad
 \bar V=\psi_k^+\text{ on }[\upsilon_{k+1},1].
\end{align*}
On each side of $p_{\mathrm c}$, the applicable line is the one
corresponding to the first inverse point in that direction that is
undefined.
Indeed, if an inverse $U_i(z)$ is undefined, then
$z\notin[a,1-b]=F_i([0,1])$, so $F_i(p)-z$ has the same sign for every
$p\in[0,1]$.  Thus no further switching point occurs before the
corresponding endpoint.  Together with the displayed cells, the
terminal assignments define $\bar V$ on $[0,1]$.  At a shared endpoint
either adjacent formula may be used provisionally; the crossing
equalities below show that the formulas agree.

Consider the rule that searches site~2 below $p_{\mathrm c}$ and site~1
above $p_{\mathrm c}$.  On each displayed cell, the next belief under
this rule lies in the interval on which the continuation function in
\eqref{eq:negative-class3-recursion} is used:
\begin{align}
F_2([\lambda_{k+1},\omega_k])
&=[\upsilon_k,\zeta_k],
\label{eq:negative-map-A}\\
F_2([\omega_{k+1},\lambda_{k+1}])
&=[\zeta_k,\upsilon_{k+1}],
\label{eq:negative-map-B}\\
F_1([\zeta_k,\upsilon_{k+1}])
&=[\omega_k,\lambda_k],
\label{eq:negative-map-C}\\
F_1([\upsilon_{k+1},\zeta_{k+1}])
&=[\lambda_{k+1},\omega_k].
\label{eq:negative-map-D}
\end{align}
The four intervals on the right-hand sides carry, respectively, the
functions $\psi_{k-1}^+$, $\phi_k^+$, $\psi_{k-1}^-$, and
$\phi_k^-$.  It follows from \eqref{eq:negative-class3-recursion} that
\begin{align*}
 \bar V&=\mathcal B_2\bar V
 &&\text{on every displayed lower cell},\\
 \bar V&=\mathcal B_1\bar V
 &&\text{on every displayed upper cell}.
\end{align*}
The same identities hold on the terminal intervals.  Indeed, the
fixed-sign observation above shows that the relevant map $F_i$ sends
each terminal interval into the continuation cell appearing in the
corresponding formula in \eqref{eq:negative-class3-recursion}.

The central equality in \eqref{eq:negative-class3-anchor} generates the
following equalities between adjacent affine pieces by repeated
application of Lemma~\ref{lem:negative-pullback}:
\begin{align*}
 \psi_{k-1}^-(\omega_k)&=\phi_k^-(\omega_k),
 &
 \phi_k^-(\lambda_{k+1})&=\psi_k^-(\lambda_{k+1}),\\
 \psi_{k-1}^+(\zeta_k)&=\phi_k^+(\zeta_k),
 &
 \phi_k^+(\upsilon_{k+1})&=\psi_k^+(\upsilon_{k+1}),
\end{align*}
for every $k\ge0$ for which the points exist.  Hence each constructed
inverse point is a switching point between its two adjacent affine
pieces, and $\bar V$ is continuous.  All the pieces are strictly
positive because $W_1,W_2>0$ and branch operations add positive
immediate costs.  Lemma~\ref{lem:negative-W-slopes} gives positive slope
to the central lower piece $\psi_{-1}^-=W_2$ and negative slope to the
central upper piece $\psi_{-1}^+=W_1$.  Every other lower piece is
obtained by applying $\mathcal B_2$ to a positive affine upper piece
with negative slope, so \eqref{eq:negative-sign2} gives it positive
slope.  Similarly, \eqref{eq:negative-sign1} gives every other upper
piece negative slope.

It remains to compare slopes within the two families.  The central
difference
\[
G_0=\psi_{-1}^+-\psi_{-1}^-
\]
is strictly decreasing and vanishes at $p_{\mathrm c}$, so it is
positive to the left of $p_{\mathrm c}$ and negative to the right.  At
$\omega_0=U_2(p_{\mathrm c})$,
\[
\psi_{-1}^--\phi_0^-
=\mathcal B_2\psi_{-1}^--\mathcal B_2\psi_{-1}^+.
\]
Lemma~\ref{lem:negative-pullback} shows that this affine difference is
positive to the left of $\omega_0$ and negative to the right.  Its
slope is therefore negative, which gives
$m(\psi_{-1}^-)<m(\phi_0^-)$.  Similarly, at
$\zeta_0=U_1(p_{\mathrm c})$,
\[
\psi_{-1}^+-\phi_0^+
=\mathcal B_1\psi_{-1}^+-\mathcal B_1\psi_{-1}^-.
\]
This difference is negative to the left of $\zeta_0$ and positive to
the right, so its slope is positive and
$m(\psi_{-1}^+)>m(\phi_0^+)$.  Repeating these pullbacks alternately
through $F_1$ and $F_2$, for as long as the inverse points exist, gives
\begin{equation}
0<m(\psi_{-1}^-)<m(\phi_0^-)<m(\psi_0^-)
<m(\phi_1^-)<m(\psi_1^-)<\cdots
\label{eq:negative-positive-slopes}
\end{equation}
and
\begin{equation}
0>m(\psi_{-1}^+)>m(\phi_0^+)>m(\psi_0^+)
>m(\phi_1^+)>m(\psi_1^+)>\cdots.
\label{eq:negative-negative-slopes}
\end{equation}
When the cells are read from left to right, the displayed slopes are
nonincreasing.  Hence $\bar V$ is concave.

For each cobweb cell, its image under the failed-search map corresponding
to the competing action can intersect at most three adjacent cobweb
cells.  Comparing the images of the endpoints gives
\begin{align}
F_1([\lambda_{k+1},\omega_k])
&\subset[\upsilon_k,\zeta_{k+1}],
&&\text{so only }\psi_{k-1}^+,\phi_k^+,\psi_k^+\text{ occur},
\label{eq:negative-inc-A}\\
F_1([\omega_{k+1},\lambda_{k+1}])
&\subset[\zeta_k,\upsilon_{k+2}],
&&\text{so only }\phi_k^+,\psi_k^+,\phi_{k+1}^+\text{ occur},
\label{eq:negative-inc-B}\\
F_2([\zeta_k,\upsilon_{k+1}])
&\subset[\omega_{k+1},\lambda_k],
&&\text{so only }\psi_k^-,\phi_k^-,\psi_{k-1}^-\text{ occur},
\label{eq:negative-inc-C}\\
F_2([\upsilon_{k+1},\zeta_{k+1}])
&\subset[\lambda_{k+2},\omega_k],
&&\text{so only }\phi_{k+1}^-,\psi_k^-,\phi_k^-\text{ occur}.
\label{eq:negative-inc-D}
\end{align}
For example, for \eqref{eq:negative-inc-C},
\[
\omega_{k+1}=U_2(\upsilon_{k+1})
<F_2(\upsilon_{k+1})
<F_1(\upsilon_{k+1})
=\omega_k
\]
and
\[
\lambda_{k+1}=U_2(\zeta_k)
<F_2(\zeta_k)
<F_1(\zeta_k)
=\lambda_k.
\]
Since $F_2$ is decreasing, its image is contained in
$[\omega_{k+1},\lambda_k]$.  The other inclusions follow from the same
endpoint inequalities and the inverse definitions
\eqref{eq:negative-ab}--\eqref{eq:negative-cd}.  If a displayed outer
endpoint has not been constructed because its inverse recursion has
terminated, the fixed-sign argument used in defining $\bar V$ shows
that the forward image stays within the last displayed interval; no
additional affine piece occurs.

Consider a $\phi_k^-$-cell.  The competing site~1 continuations are
$\psi_{k-1}^+,\phi_k^+,\psi_k^+$.  For $k\ge1$, the resulting affine
subpieces are
\[
\mathcal B_1^2\phi_{k-1}^-,
\qquad
\mathcal B_1^2\psi_{k-1}^-,
\qquad
\mathcal B_1^2\phi_k^-.
\]
The underlying slopes are positive and, by
\eqref{eq:negative-positive-slopes}, are at most $m(\phi_k^-)$.
Lemma~\ref{lem:negative-flattening} therefore shows that every
affine subpiece $g$ of the competing branch $\mathcal B_1\bar V$ on
a $\phi_k^-$-cell satisfies
\begin{equation}
m(g)<m(\phi_k^-).
\label{eq:negative-wrong-A}
\end{equation}
For $k=0$, the additional central subpiece is
$\mathcal B_1\psi_{-1}^+=W_1$, which already has negative slope.
On a $\psi_k^-$-cell, the continuation pieces in
\eqref{eq:negative-inc-B} give the subpieces
\[
 \mathcal B_1^2\psi_{k-1}^-,
 \qquad
 \mathcal B_1^2\phi_k^-,
 \qquad
 \mathcal B_1^2\psi_k^-.
\]
Their underlying slopes are positive and no larger than
$m(\psi_k^-)$ by \eqref{eq:negative-positive-slopes}.
Lemma~\ref{lem:negative-flattening} therefore shows that every affine
subpiece $g$ of $\mathcal B_1\bar V$ on a $\psi_k^-$-cell satisfies
\begin{equation}
m(g)<m(\psi_k^-).
\label{eq:negative-wrong-B}
\end{equation}

On a $\phi_k^+$-cell, the continuation pieces in
\eqref{eq:negative-inc-C} give, for $k\ge1$, the subpieces
\[
 \mathcal B_2^2\phi_k^+,
 \qquad
 \mathcal B_2^2\psi_{k-1}^+,
 \qquad
 \mathcal B_2^2\phi_{k-1}^+.
\]
For $k=0$, the last subpiece is the central fixed-point function
$\mathcal B_2W_2=W_2$.  Apart from that positive-slope function, the
underlying slopes are negative and no smaller than $m(\phi_k^+)$ by
\eqref{eq:negative-negative-slopes}.  Lemma~\ref{lem:negative-flattening}
therefore shows that every affine subpiece $g$ of
$\mathcal B_2\bar V$ on a $\phi_k^+$-cell satisfies
\begin{equation}
m(g)>m(\phi_k^+).
\label{eq:negative-wrong-C}
\end{equation}
On a $\psi_k^+$-cell, the continuation pieces in
\eqref{eq:negative-inc-D} give the subpieces
\[
 \mathcal B_2^2\psi_k^+,
 \qquad
 \mathcal B_2^2\phi_k^+,
 \qquad
 \mathcal B_2^2\psi_{k-1}^+.
\]
The same slope ordering and two-step comparison show that every affine
subpiece $g$ of $\mathcal B_2\bar V$ on a $\psi_k^+$-cell satisfies
\begin{equation}
m(g)>m(\psi_k^+).
\label{eq:negative-wrong-D}
\end{equation}

The two central cells require the same comparison but are not included
in \eqref{eq:negative-inc-A}--\eqref{eq:negative-inc-D}.  The inverse
definitions and the self-cobweb inequalities give
\[
 F_1([\omega_0,p_{\mathrm c}])\subset[p_{\mathrm c},\upsilon_1],
 \qquad
 F_2([p_{\mathrm c},\zeta_0])\subset[\lambda_1,p_{\mathrm c}].
\]
These inclusions hold when the displayed outer endpoints exist.  If an
outer endpoint does not exist, the same fixed-sign argument shows that
the image remains in the corresponding terminal interval, and the
comparisons below are unchanged.
Thus, on the central lower cell, where the active piece is
$\psi_{-1}^-=W_2$, the competing site~1 branch has subpieces
$W_1=\mathcal B_1W_1$ and
$\mathcal B_1^2W_2$.  Both have slope smaller than $m(W_2)$: the first
by Lemma~\ref{lem:negative-W-slopes}, and the second by
Lemma~\ref{lem:negative-flattening}.  On the central upper cell, where
the active piece is $\psi_{-1}^+=W_1$, the competing site~2 branch has
subpieces $W_2=\mathcal B_2W_2$ and $\mathcal B_2^2W_1$.  Both have
slope greater than $m(W_1)$, by the same two lemmas.

Define
\[
\mathcal E_1:=\mathcal B_1\bar V-\bar V\qquad(p\le p_{\mathrm c}),
\qquad
\mathcal E_2:=\mathcal B_2\bar V-\bar V\qquad(p\ge p_{\mathrm c}).
\]
They are continuous and piecewise affine, and the preceding slope
inequalities give $\mathcal E_1'<0$ on every lower affine subcell and
$\mathcal E_2'>0$ on every upper affine subcell, including the two central
cells just checked.  Since $p_{\mathrm c}<P_1$,
Lemma~\ref{lem:negative-self-cobweb} places $F_1(p_{\mathrm c})$ in the
central $\psi_{-1}^+=W_1$ cell; since $p_{\mathrm c}>P_2$, it places
$F_2(p_{\mathrm c})$ in the central $\psi_{-1}^-=W_2$ cell.  Using
$W_i=\mathcal B_iW_i$ and \eqref{eq:negative-class3-anchor},
\[
\mathcal E_1(p_{\mathrm c})=\mathcal E_2(p_{\mathrm c})=0.
\]
Hence
\begin{equation}
\mathcal E_1(p)\ge0\quad(p<p_{\mathrm c}),
\qquad
\mathcal E_2(p)\ge0\quad(p>p_{\mathrm c}),
\label{eq:negative-class3-excess}
\end{equation}
and again $\bar V=\mathcal T\bar V$ with site~2 minimising below
$p_{\mathrm c}$ and site~1 above it.  Consequently
$p_{\mathrm c}$ is a threshold for this minimising selector.  The
strict inequalities in Lemma~\ref{lem:negative-interlace}, together with
Lemma~\ref{lem:negative-finite-cobweb}, show that the construction has
only finitely many cells.  Every affine
piece is obtained from $W_1$ or $W_2$ by finitely many branch
operations; it is therefore positive.  Hence $\bar V$ is bounded and
positive on $[0,1]$.

\subsection{Verification and completion of the proof}

The cobweb constructions produce bounded positive Bellman fixed points.
The next elementary verification identifies them with the optimal value.

\begin{lemma}[Verification for $\Delta<0$]
\label{lem:negative-verification}
Let $\bar V$ be a bounded nonnegative solution of
$\bar V=\mathcal T\bar V$ under \eqref{eq:negative-strict}, and let a
stationary policy $\mu$ select a minimising action at every belief.  Then,
for every initial belief $p$, its expected total cost is $\bar V(p)$,
and $\bar V(p)=V(p)$.
\end{lemma}

\begin{proof}
After the first unsuccessful search and movement, every belief lies in
$[a,1-b]$.  Hence at every subsequent epoch the probability of site~1
is at least $a$ and the probability of site~2 is at least $b$.  Whatever
action is chosen, the conditional probability of detection is therefore
at least
\[
\varepsilon
:=\min\{a(1-\alpha_1),\,b(1-\alpha_2)\}>0.
\]
Thus the probability of surviving $n$ further searches under any policy
is at most $(1-\varepsilon)^n$.

For an arbitrary policy $\pi$, let
$\tau\in\{1,2,\ldots\}\cup\{\infty\}$ be the epoch at which detection
occurs, with $\tau=\infty$ if detection never occurs, and set
\[
 \tau_n:=\min\{n,\tau\}.
\]
Let $\iota_k$ be the site searched at epoch $k$, and let
$\mathcal S_n:=\{\tau>n\}$ be the event that the first $n$ searches are
all unsuccessful.  On $\mathcal S_n$, write
$p_{n+1}$ for the belief after the resulting movement.  Iterating the
Bellman inequalities for $\bar V$ along this history gives
\[
\bar V(p)
\le
\EE_\pi\!\left[\sum_{k=1}^{\tau_n}C_{\iota_k}\right]
+
\EE_\pi[\1_{\mathcal S_n}\bar V(p_{n+1})].
\]
After the first search, the geometric bound above gives
\[
\PP_\pi(\mathcal S_n)\le(1-\varepsilon)^{n-1}\qquad(n\ge1).
\]
Since $\bar V$ is bounded and nonnegative,
\[
0\le
\EE_\pi[\1_{\mathcal S_n}\bar V(p_{n+1})]
\le
\|\bar V\|_\infty(1-\varepsilon)^{n-1}
\longrightarrow0.
\]
By definition, $\tau_n\uparrow\tau$, and
\[
\sum_{k=1}^{\tau_n}C_{\iota_k}
\uparrow
\sum_{k=1}^{\tau}C_{\iota_k}.
\]
The monotone convergence theorem therefore gives
\[
\bar V(p)
\le
\EE_\pi\!\left[\sum_{k=1}^{\tau}C_{\iota_k}\right].
\]
Taking the infimum over all policies $\pi$ yields
$\bar V(p)\le V(p)$.

For the policy $\mu$, the selected action attains the Bellman minimum,
so the iterated relation is an equality:
\[
\bar V(p)
=
\EE_\mu\!\left[\sum_{k=1}^{\tau_n}C_{\iota_k}\right]
+
\EE_\mu[\1_{\mathcal S_n}\bar V(p_{n+1})].
\]
Applying the two limits above under $\mu$ gives
\[
\bar V(p)
=
\EE_\mu\!\left[\sum_{k=1}^{\tau}C_{\iota_k}\right]
\ge V(p).
\]
Thus $\bar V(p)=V(p)$ for every $p\in[0,1]$, and $\mu$ is optimal.
\end{proof}

\begin{proof}[Proof of Theorem~\ref{thm:negative}]
Proposition~\ref{prop:negative-three} exhausts all cases in which
$G_0(P_2)$ and $G_0(P_1)$ are nonzero.  In case~(i), if
$G_1\le0$ on $[0,1]$, then $W_1$ is a Bellman fixed point with site~1
minimising everywhere.  In case~(iii), if $G_2\ge0$ on $[0,1]$, then
$W_2$ is a Bellman fixed point with site~2 minimising everywhere.
Lemma~\ref{lem:negative-verification} proves optimality in both cases.
The remaining possibility in case~(i) is proved in
Section~\ref{subsec:negative-one-sided}; the upper one-sided case follows
by site exchange; and the middle case is proved in
Section~\ref{subsec:negative-interlaced}.
Lemma~\ref{lem:negative-verification} identifies each constructed Bellman
fixed point with $V$.

It remains only to close the two boundaries between the three cobweb
configurations,
$G_0(P_2)=0$ and $G_0(P_1)=0$.  For $c>0$, let $\theta(c)$ be the
parameter vector obtained by replacing $C_1$ with $c$ and leaving all
other parameters fixed.  Write $W_1^{(c)}$ for the cost of always
searching site~1 under $\theta(c)$.  If $h_1(p)$ denotes the expected
number of searches under this policy, then
\[
 W_1^{(c)}(p)=c h_1(p),
\]
where $0<h_1(p)<\infty$ and $h_1$ does not depend on $c$.  The function
$W_2$ and the fixed points $P_1,P_2$ also do not depend on $c$.
Set $G_0^{(c)}:=W_1^{(c)}-W_2$.  Consequently, for
$i\in\{1,2\}$,
\[
 G_0^{(c)}(P_i)=c h_1(P_i)-W_2(P_i),
\]
and the equation $G_0^{(c)}(P_i)=0$ has the unique solution
\[
 \widehat c_i:=\frac{W_2(P_i)}{h_1(P_i)}>0.
\]

Now fix a boundary value $c_0>0$ for which
$G_0^{(c_0)}(P_2)=0$ or $G_0^{(c_0)}(P_1)=0$.  Choose
\[
 c^{(j)}>0,
 \qquad
 c^{(j)}\longrightarrow c_0,
 \qquad
 c^{(j)}\notin\{\widehat c_1,\widehat c_2\}
 \quad(j\ge1).
\]
For every $j$, Proposition~\ref{prop:negative-three} and the cobweb
verification above imply that $D_{\theta(c^{(j)})}$ has no reverse strict
crossing; explicitly, for all $p,p'\in[0,1]$,
\[
 p<p',
 \qquad
 D_{\theta(c^{(j)})}(p)<0
 \quad\Longrightarrow\quad
 D_{\theta(c^{(j)})}(p')\le0.
\]
Moreover, Lemma~\ref{lem:D-parameter-continuity} gives, for each
$x\in[0,1]$,
\begin{equation}
 D_{\theta(c^{(j)})}(x)\longrightarrow D_{\theta(c_0)}(x).
 \label{eq:negative-boundary-D-limit}
\end{equation}
Suppose, to the contrary, that there were $p<p'$ such that
\[
 D_{\theta(c_0)}(p)<0<D_{\theta(c_0)}(p').
\]
Set
\[
 \varepsilon
 :=\frac12\min\{-D_{\theta(c_0)}(p),D_{\theta(c_0)}(p')\}>0.
\]
By \eqref{eq:negative-boundary-D-limit}, for all sufficiently large $j$,
\[
 \left|D_{\theta(c^{(j)})}(p)-D_{\theta(c_0)}(p)\right|<\varepsilon,
 \qquad
 \left|D_{\theta(c^{(j)})}(p')-D_{\theta(c_0)}(p')\right|<\varepsilon.
\]
Hence
\[
 D_{\theta(c^{(j)})}(p)<0<D_{\theta(c^{(j)})}(p'),
\]
contradicting the no-reverse-crossing property at $c^{(j)}$.  Thus that
property also holds at the two boundary values.  The supremum construction
in \eqref{eq:infinite-threshold}, and the argument following that
definition, therefore give a threshold minimising selector.
Proposition~\ref{prop:verification} verifies its optimality.  This
completes the strict negative-determinant case.
\end{proof}

\begin{remark}[What the classical cobweb algebra is measuring]
\label{rem:negative-interpretation}
The three configurations correspond exactly to the three possible
locations of the central threshold relative to $P_2<P_1$.  The
fixed-point identities
\eqref{eq:negative-G1-match}--\eqref{eq:negative-G2-match} show that
their boundaries are precisely
the Bellman consistency conditions at $P_2$ and $P_1$.  Once the central
crossing is fixed, every later switching point is obtained by taking its
inverse image under a finite composition of the failed-search maps.
Thus the apparently separate coefficient calculations in the classical
cobweb construction are all encoded by the crossing-transport identity
in Lemma~\ref{lem:negative-pullback}.
\end{remark}

\end{document}